\documentclass{amsart}
\usepackage{amsmath}
\usepackage{amsthm}
\usepackage{amssymb}
\usepackage{mathrsfs}
\usepackage{graphicx}
\usepackage{caption}
\usepackage{subcaption}
\usepackage{xcolor}
\usepackage{soul}
\usepackage{float}

\newtheorem{theorem}{Theorem}[section]
\newtheorem{lemma}[theorem]{Lemma}

\theoremstyle{definition}
\newtheorem{definition}[theorem]{Definition}

\newtheorem{corollary}[theorem]{Corollary}

\newtheorem{propn}[theorem]{Proposition}
\theoremstyle{remark}
\newtheorem{remark}[theorem]{Remark}

\numberwithin{equation}{section}
\newcommand{\R}{\mathbb{R}}

\newcommand{\F}{\mathscr{F}}
\newcommand{\Za}{\mathscr{Z}}
\newcommand{\N}{\mathbb{N}}
\newcommand{\Z}{\mathbb{Z}}
\newcommand{\Hi}{\mathcal{H}}
\newcommand{\DF}{\mathcal{F}}

\begin{document}


\title{The system of translates and the special affine Fourier transform}
\author{Md Hasan Ali Biswas}
\address{Department of Mathematics, Indian Institute of Technology Madras, Chennai - 600036, India}
\email{mdhasanalibiswas4@gmail.com}

\author{Frank Filbir}
\address{Mathematical Imaging and Data Analysis, Helmholtz Center Munich,
\newline Department of Mathematics, Technical University of Munich, Germany}
\email{filbir@helmholtz-muenchen.de, frank\_filbir@web.de}


\author{Radha Ramakrishnan$^*$}
\address{Department of Mathematics, Indian Institute of Technology Madras, Chennai - 600036, India}
\email{radharam@iitm.ac.in}



\keywords{Convolution, Poisson summation formula, shift invariant space, Wendel's theorem, Zak-transform.\\
*Corresponding author: Radha Ramakrishnan, email:radharam@iitm.ac.in}

\begin{abstract}
The translation operator $T^A$ associated with the special affine Fourier transform (SAFT) $\F_A$ is introduced from harmonic analysis point of view. \textcolor{black}{The analogues of Wendel's theorem, Wiener theorem, Weiner-Tauberian theorem and Bernstein type inequality in the context of the SAFT are established.} The shift invariant space $V_A$ associated with the special affine Fourier transform is introduced and studied along with sampling problems.
\end{abstract}

\maketitle

\section{Introduction and background}
The special affine Fourier transform (SAFT) was first considered by S. Abe and J. T. Sheridan in \cite{abe} for the study of certain operations on optical wave 
functions. The SAFT is formally defined as  
\begin{equation}\label{eq:Intro2}
\mathscr{F}_Af(\omega)=\frac{1}{\sqrt{2\pi |b|}}\int_\R f(t) e^{\frac{i}{2b}(at^2+2pt-2\omega t+d\omega^2+2(bq-dp)\omega)} dt,
\end{equation}
where $A$ stands for the set $\{a,b,c,d,p,q\}$ of real parameters which satisfy the relation $ad-bc=1$. The integral transform \eqref{eq:Intro2} is related to the special affine linear transform of the phase space 
\begin{equation}\label{eq:Intro1}
\begin{pmatrix}
t^\prime\\\omega^\prime
\end{pmatrix}
=\begin{pmatrix}
a&b\\
c&d
\end{pmatrix}
\begin{pmatrix}
t\\ \omega
\end{pmatrix}
+
\begin{pmatrix}
p \\ q
\end{pmatrix}.
\end{equation} 
Due to the conditions on the parameters $a,b,c,d$ the matrix in \eqref{eq:Intro1} belongs to the special linear group $\mathrm{SL}(2,\R)$ and the affine transform \eqref{eq:Intro1} is therefore given by elements from the inhomogeneous linear group 
\[
\mathrm{ISL}(2,\R)=\left\{{\footnotesize \left(\begin{array}{cc} M& v\\ 0&1\end{array}\right)}: \, M\in\mathrm{SL}(2,\R),\ v\in\R^2\right\}.
\]
This justifies the name special affine Fourier transform for \eqref{eq:Intro2}. The action of the group $\mathrm{SL}(2,\R)$ on the time frequency plane and the relation to quadratic Fourier transforms is well studied. We will not go into the details here but refer to the book \cite{Gosson_symplectic}.
\par A number of important transforms are special cases of the SAFT. For example, $A=\{0,1,-1,0,0,0\}$ gives the ordinary Fourier transform and $A=\{0,-1,1,0,0,0\}$ its inverse. The parameter set $A=\{\cos\theta, \sin\theta,-\sin\theta, \cos\theta, 0,0\}$ gives the fractional Fourier transform, and 
$A=\{1,\lambda,0,1,0,0\}$ produces the Fresnel transform. 
\par In optics, certain one parameter subgroups of the $\mathrm{ISL}(2,\R)$ are of special interest. Among them are the fractional Fourier transform, the Fresnel transform (also called free space propagation in this context), the hyperbolic transform $A=\{\cosh\theta,\sinh\theta,\sinh\theta,\cosh\theta,0,0\}$, the lens transform $A=\{1,0,\lambda,1,0,0\}$, and the magnification transform $A=\{e^\beta,0,0,e^{-\beta},0,0\}$. The latter two cases need a careful analysis for the limit case $b\to 0$ which we will not consider in this paper. We will not try to expound the various connections to optics further but refer to \cite{saftOptics} for more details.
\par In this paper, we consider \eqref{eq:Intro2} from the point of view of applied harmonic analysis and take it as a signal transform of a (suitable) function. We are mainly interested in studying the principal shift invariant spaces and sampling theorems related to the SAFT. In the classical case, the principal shift invariant space generated by $\phi\in L^2(\R)$ 
is defined as $V(\phi)=\overline{span}\{T_k\phi=\phi(\cdot-k): k\in\Z\}$. The classical Fourier transform ($A=\{0,1,-1,0,0,0\}$ in \eqref{eq:Intro2}) plays a crucial role for the analysis of such spaces. The crucial point is that the ordinary translation and the classical Fourier transform are intimately related. This is due to the identity $e^{i\omega (t-x)}=e^{i\omega t} e^{-i\omega x}$ which gives, as a consequence the convolution theorem, the relation between translation, modulation, Fourier transform etc. These theorems are used over and over again in Fourier analysis and in the study of shift invariant spaces in particular. It is completely obvious that the ordinary translation does not interact nicely with the SAFT (resp. its complex exponential kernel). Hence working with the SAFT a new concept of a translation is needed. This generalized translation should be linked to the SAFT in an analogous manner as the ordinary translation is linked to the classical Fourier transform. If this is the case, then it seems reasonable to expect that the central theorems (convolution theorem etc.) hold in a similar manner. An idea for the construction of such generalized translation comes from the observation that in the classical setting we have 
\[
T_xf(t)=\mathscr{F}^{-1}(e^{-i \omega x}\mathscr{F}f)(t).
\]
In this paper, we define a new translation operator $T^A_x$, which serves our purpose in the case of SAFT. 
\par In \cite{Bhandari} A. Bhandari and A.I. Zayed considered chirp-modulation and used this to obtain a convolution theorem. However, they did not define a generalized translation operator explicitly, and hence did not investigate the consequences of this concept with respect to harmonic analysis. We shall demonstrate that the generalized translation is the suitable concept to obtain analogues of fundamental theorems such as Wendel's theorem for the multipliers, \textcolor{black}{Wiener theorem and Weiner-Tauberian theorem in connection with the closed ideals of translation invariant spaces in the context of the SAFT.} For a study of multipliers and Wendel's theorem for the Fourier transform we refer to \cite{larsen}, 
\textcolor{black}{ for Wiener theorem and Wiener-Tauberian theorem we refer to \cite{Folland1995} and \cite{Rudin_FA}}. The novelty of this approach is that apart from these theorems, one can look into the study of multiplier theory, including H\"{o}rmander multiplier theorem in the SAFT domain. Moreover, one can define an appropriate modulation operator in connection with the SAFT, using which one can define modulation spaces associated with the SAFT. This in turn, motivates to study multiplier results for the new modulation spaces. (See \cite{H3} in this connection.) 
\par Using the new translation operator $T^A_x$ the $A$-shift invariant spaces are defined as $V_A(\phi)=\overline{span}\{T^A_k\phi: k\in\Z\}$ for an appropriate function $\phi\in L^2(\R)$. When $\phi$ belongs to the Wiener amalgam space $W\big(C,\ell^1(\Z)\big)$ the space $V_A(\phi)$ turns out to be a reproducing kernel Hilbert space. Moreover, we give characterization theorems for the system of translates $\{T^A_k\phi: k\in\Z\}$ to be a \textcolor{black}{frame sequence}, orthogonal system or a Riesz sequence.
If the system of translates is a frame then an important question is about the nature of the dual frame elements. We show that in our setting, the elements of the dual frame of system of $A$-translates are also $A$-translates of a single function. For a study of shift invariant spaces, system of translates, frames and Riesz basis in the classical case, we refer to Christensen \cite{ole}.\\
\par In the final part of the paper we study the sampling in $A$-shift invariant spaces. A fundamental problem in sampling theory is to find, for a certain class of 
functions, appropriate conditions on a countable sampling set $X=\{x_j\in\R: j\in J\}$ under which a given function $f\in V$ can be reconstructed 
uniquely and stably from the samples $\{f(x_j): j\in J\}$. We refer to the work of Butzer and Stens \cite{Butzer} for a review on sampling theory and its history. When $V$ is the classical principal shift invariant space with a single generator $V(\phi)$ or multi-generators $V(\phi_1, \phi_2,...,\phi_r),~\phi, \phi_1, \phi_2,... \phi_r\in L^2(\R)$, there is a huge literature available on several interesting problems connected with sampling theory starting from the fundamental Shannon sampling theorem. We cite only a few references in this connection for the reader to get familiarity with this subject matter. (See \cite{AldGro2000}, \cite{grochenig}, \cite{Aldroubi2}, \cite{Aldroubi}, \cite{IEEE}, \cite{Xian1}, \cite{Garcia}, \cite{Kadec}, \cite{Nashed}, \cite{sarvesh}, \cite{Sun}, \cite{Xian}, \cite{Zhang}, \cite{ZhouSunJFAA}). \\
\par In this paper, similar to Theorem $4.2$ of \cite{grochenig} and Theorem $2.1$ of \cite{Antony}, we obtain equivalent conditions for a set $X$ to be a stable set of sampling for $V_A(\phi)$ in terms of the operator $U$ where $U_{jk}=T_k^A\phi(x_j)$, the reproducing kernel and the Zak transform $Z_A$, which we introduce here. We also obtain a sufficient condition for the set of integers to be a stable set of sampling for $V_A(\phi).$
\textcolor{black}{In the study of non uniform sampling and average sampling, Bernstein type inequalities play an important role. In this paper we obtain an analogue of Bernstein type inequality for $V_A(\phi)$. However, we do not intend to study non-uniform sampling and average sampling in this paper. We focus on uniform sampling. In particular,} when $\Z$ turns out to be a stable set of sampling, we obtain a reconstruction formula and hence a sampling theorem \textcolor{black}{in the sense of $L^2$ convergence} for certain $A$-shift invariant spaces $V_A(\phi)$. \textcolor{black}{Further, under some additional hypotheses on $\phi$, we obtain a sampling formula in the sense $L^2$ convergence and uniform convergence.} As corollaries we obtain Shannon sampling theorem in the SAFT domain and sampling theorem for the $A$-shift invariant space generated by second order $B$-spline.\\

\par We organize our paper as follows. In section $2$, we define $A$-convolution of a measure and a function and chirp modulation $C_s\mu$ of a measure $\mu$. Using this and the function $\rho_A(x)=e^{\frac{i}{2b}(ax^2+2px)}$, we obtain a relation between classical translation and $A$-translation. We prove an analogue of Wendel's theorem for the SAFT. \textcolor{black}{In section $3$, we study closed ideals in the Banach algebra $(L^1(\R),\star_A)$. We obtain analogues of Wiener theorem and Wiener-Tauberian theorem in the context of the SAFT.} In section $4$, we study $A$-shift invariant spaces and their theoretical aspects. Section $5$ as well as Section $6$ are devoted to sampling theorems in $A$-shift invariant spaces. Finally in Section $7$, we present a local reconstruction method for sampling in $A$-shift invariant spaces along with implementation.\\ 
\par Now we shall provide the necessary terminology and background for this paper.



 \par Let $0\neq \Hi$ be a separable Hilbert space.

\begin{definition}
A sequence $\{f_k:k\in \N\}$ of elements in $\Hi$ is a frame for $\Hi$ if there exist $m,M>0$ such that
$$
m\|f\|^2\leq \sum_{k=1}^\infty |\langle f,f_k \rangle |^2\leq M\|f\|^2,\qquad ~~~ f\in \Hi.
$$
\end{definition}
\par The numbers $m,M$ are called frame bounds. If we have the right hand side inequality for a sequence in $\Hi$, then that sequence is called a Bessel sequence.

\begin{definition}
Let $\{f_k:k\in \N\}$ be a Bessel sequence in $\Hi$, then the synthesis operator $T: \ell ^2 \to \Hi$ is defined by
$$
T(\{c_k\})=\sum_{k=1}^\infty c_kf_k,\quad \{c_k\} \in \ell ^2.
$$
\par The adjoint of $T$ is given by $T^*(f)=\{\langle f, f_k \rangle \},~f\in \Hi$, called the analysis operator. Composing $T$ and $T^*$ we obtain the frame operator $$S:\Hi \to \Hi, ~ S(f)=\sum_{k=1}^\infty \langle f,f_k\rangle f_k.$$ 
The operator $S$ is invertible. Further if $\{f_k:k\in \N\}$ is a frame for $\Hi$, then $\{S^{-1}f_k:k\in \N\}$ is also a frame for $\Hi$ and it is called the canonical dual frame of the frame $\{f_k:k\in \N\}.$
\end{definition}

\begin{definition}
A sequence $\{f_k:k\in \N\}$ in $\Hi$ is said to be a Riesz basis if there exist a bounded invertible operator $T$ on $\Hi$ and an orthonormal basis $\{u_k:k\in \N\}$ of $\Hi$ such that $f_k=Tu_k,~\forall ~ k\in \N.$ The sequence $\{f_k:k\in \N\}$ is called a Riesz sequence if it is a Riesz basis for its closed linear span.
\par Equivalently $\{f_k:k\in \N \}$ is a Riesz sequence if there exist $m,M>0$ such that
\[
m\|c\|^2\leq \|\sum_{k=1}^\infty c_kf_k\|^2 \leq M\|c\|^2,~~ \text{for every finite sequence} \  \{c_k\}.
\]
\end{definition}

\begin{definition}
Let $\{f_k:k\in \N\}$ be a Riesz basis for $\Hi$. The dual Riesz basis of $\{f_k:k\in \N\}$ is the unique sequence $\{g_k:k\in \N\}$ in $\Hi$ satisfying 
\[
f=\sum_{k=1}^\infty \langle f, g_k \rangle f_k,\quad \forall f\in \Hi.
\]
\end{definition}

\begin{definition}
The Gramian associated with the Bessel sequence $\{f_k:k\in \N\}$ is an operator on $\ell ^2$ whose $jk^{th}$ entry in the matrix representation with respect to the canonical orthonormal basis is $\langle f_k,f_j\rangle .$
\end{definition}

\par It is well known that a sequence $\{f_k:k\in \N\}$ is a Riesz sequence if there exist $m,M>0$ such that its Gramian $G$ satisfy the following inequality:
$$
m\|c\|^2\leq \sum_{k\in \N} | \langle Gc_k, c_k \rangle |^2 \leq M\|c\|^2,~~\forall ~ c= \{c_k\}\in \ell ^2.
$$

\begin{definition}
A closed subspace $V$ in $L^2(\R)$ is said to be a shift invariant space if $f\in V \Rightarrow T_kf \in V,~\forall k \in \Z, ~f\in V,$ where $T_kf(t)=f(t-k).$
\par In particular, for $\phi \in L^2(\R),~V(\phi)=\overline{span}\{T_k\phi :k \in \Z\}$ is called the principal shift invariant space.
\end{definition}

\begin{definition}
A set $X=\{x_k\in \R:k\in \Z\}$ is said to be a stable set of sampling for a closed subspace $V$ of $L^2(\R)$ if there exist constants $0<m\leq M<\infty$ such that
$$
m\|f\|^2\leq \sum_{k\in \Z}|f(x_k)|^2\leq M\|f\|^2,
$$
for every $f\in V.$
\end{definition}

\begin{definition}
The Wiener amalgam space $W\big(C,\ell^p(\Z)\big)$, $1\leq p<\infty$ is defined as
\begin{equation*}
W\big(C,\ell^p(\Z)\big)=\{f\in C(\R):\sum_{k\in \Z} max_{x\in [0,1]}|f(x+k)|^p<\infty\}.
\end{equation*}
\end{definition}


\section{The new translation}
In this section, we introduce $A$-translation operator in connection with the SAFT. Using $A$-translation operator, we define $A$-convolution of a regular Borel measure and a function. Further, we obtain an analogue of Wendel's theorem. Towards this end, first we extend the definition of SAFT to the space of all regular Borel measures.


\begin{definition}
Let $f\in L^1(\R)$. Then the special affine Fourier transform is defined as 
\begin{equation}\label{eq:SAFT0}
\mathscr{F}_Af(\omega)=\frac{1}{\sqrt{2\pi |b|}}\int_\R f(t) e^{\frac{i}{2b}(at^2+2pt-2t\omega+d\omega^2+2(bq-dp)\omega)} d t,~  \omega\in\R,
\end{equation}
where $A$ stands for the set of six parameters $\{a,b,c,d,p,q\}\subset\R$ with $ad-bc=1$ and $b\neq 0$. 
\end{definition}
With the help of the following auxiliary functions
\begin{align}
    \eta_A(\omega)&=e^{\frac{i}{2b}\left(d\omega^2+2(bq-dp)\omega\right)},\\
    \rho_A(t)&=e^{\frac{i}{2b}(at^2+2pt)},
\end{align}
the SAFT can be expressed as
\begin{equation}\label{eq:SAFT1}
    \F_Af(\omega)=\frac{\eta_A(\omega)}{\sqrt{|b|}}\rho_A(f)(\omega/b)
\end{equation}
Since $|\eta_A(\omega)|=1=|\rho_A(t)|$ 
for all $t,\omega\in\R$, we immediately see from \eqref{eq:SAFT1} that $\mathscr{F}_A f\in C_0(\R)$ with $\|\mathscr{F}_Af\|_\infty \leq(2\pi |b|)^{-1/2}\, \|f\|_1$. 
Moreover, \eqref{eq:SAFT1} also shows that $\mathscr{F}_A$ can be extended to $L^2(\R)$ and defines a unitary operator on that space. In particular, 
$$
\|\mathscr{F}_Af\|_2=\|f\|_2. 
$$ 
The inverse of $\mathscr{F}_A$ on $L^2(\R)$ can also be easily determined using \eqref{eq:SAFT1} 
$$
\mathscr{F}_A^{-1}f(t)=\frac{\overline{\rho_A}(t)}{\sqrt{2\pi |b|}}\int_\R f(\omega) \overline{\eta_A}(\omega)\, e^{i \omega t/b} d\omega. 
$$
Finally, \eqref{eq:SAFT1} also provides an extension of the SAFT to $M(\R)$, the space of all complex valued bounded regular Borel measures on $\R$, equipped with the 
total variation norm. For $\mu\in M(\R)$, we have 
$$
\mathscr{F}_A\mu(\omega)= 
\frac{1}{\sqrt{2\pi |b|}}\int_\R e^{\frac{i}{2b}(at^2+2pt-2t\omega+d\omega^2+2(bq-dp)\omega)} d\mu(t) ,\,  \omega\in\R,
$$
with $\mathscr{F}_A\mu\in C_b(\R)$ and $\|\mathscr{F}_A\mu\|_\infty\leq (2\pi|b|)^{-1/2} \|\mu\|_{M(\R)}$.\\
\par We now introduce the generalized translation operator associated with SAFT. In order to do so, we fix the following notation
\begin{align}
T_xf(t)&=f(t-x). \\
 M_\omega f(t)&=e^{i\omega t}f(t)\label{eq:cmodulation}.
\end{align}

\begin{definition}
Let $x\in \R$ and $f:\R \to \R$ be a function. Then $A$-translation of $f$ by $x$, denoted by $T_x^Af$, is defined as
$$
T_x^Af(t)= T_xM_{-ax/b}f(t)= e^{-i\frac{a}{b}x(t-x)}f(t-x),~t \in \R. 
$$
\end{definition}

\par It is easy to see that $T_x^A$ is norm preserving in all spaces $L^p(\R),~1\leq p\leq \infty$ or $C_0(\R)$. \par We can relate our new translation and the classical translation in the following way, using \textcolor{black}{$\rho_A$} and the chirp modulation operator $C_\frac{a}{b},$ 
\textcolor{black}{where
\begin{equation}
C_sf(t)=e^{i\frac{s}{2}t^2}f(t).
\end{equation}}
\begin{align}
\label{chirpTranslation} C_\frac{a}{b}T_x^Af=e^{i\frac{a}{2b}x^2}T_x(C_\frac{a}{b}f)\\
\label{rhoT_x^A}\rho_AT_x^Af=\rho_A(x)T_x(\rho_Af).
\end{align}
In fact,\begin{align*}
C_\frac{a}{b}T_x^Af(t)=e^{i\frac{a}{2b}t^2}e^{-i\frac{a}{b}x(t-x)}f(t-x)=&e^{i\frac{a}{2b}x^2}e^{i\frac{a}{2b}(t-x)^2}f(t-x)\\=&e^{i\frac{a}{2b}x^2}T_x(C_\frac{a}{b}f)(t).
\end{align*}
Similarly one can show (\ref{rhoT_x^A}).\\
\par Now, we collect the properties of $T_x^A.$ 

\begin{propn}\label{pro:SAFT1}
We have the following
\begin{itemize}
\item[(i)] $T_x^AT_y^A=e^{-i\frac{a}{b}xy}T_{x+y}^A,\quad x,y\in \R.$
\item[(ii)] $(T_x^A)^*=e^{-i\frac{a}{b}x^2}T_{-x}^A,\quad x\in \R.$
\item[(iii)] Let $\chi_\omega(t)=\overline{\rho_A}(t)\, e^{i\omega t/b}$. Then $T^A_x\chi_\omega(t)=\overline{\chi_\omega}(x)\, \chi_\omega(t)$. 
\item[(iv)] Let $f\in L^1(\R)$. Then
\begin{equation}\label{gtsaft}
\mathscr F_A (T_x^Af)(\omega)= \rho_A(x)e^{-ixw/b}\mathscr F_Af(\omega),\quad x,\omega \in \R.
\end{equation}
\end{itemize}
\end{propn}

\begin{proof}
The proof is straightforward.
\end{proof}
%

\par It is interesting to note that the map $x\mapsto T_x$ is a group representation, whereas from Proposition \ref{pro:SAFT1} (i) it follows that $x\mapsto T_x^A$ is just a projective representation in general, which shows that the new translation is fundamentally different from that of the classical translation. Proposition \ref{pro:SAFT1} (iii) is what is known in harmonic analysis as a product formula. The relation (\ref{gtsaft}) extends to functions $f\in L^2(\R)$ as well. For those functions we have in particular 
\begin{equation}\label{gtsaft2}
T_x^Af=\F_A^{-1}(\rho_A(x)e^{-ix\cdot /b}\F_Af),
\end{equation}
where the equality \textcolor{black}{holds in the sense of $L^2(\R)$ functions.}\\
\begin{definition}
Let $\mu\in M(\R)$ and $s\in \R$, then $C_s\mu$ is defined by $d(C_s\mu)(x)=e^{\frac{i}{2}sx^2}d\mu(x).$ 
\end{definition}
Clearly, $C_s\mu\in M(\R).$ Similarly, we can define $\rho_A\mu$ as $d(\rho_A\mu)(x)=\rho_A(x)d\mu(x)$. 
\par Using the $A$-translation, we define the $A$-convolution of  $\mu\in M(\R)$ and $f\in L^1(\R)$ as 
\begin{equation}\label{eq:SAFT4}
(\mu\star_A f)(t) =\frac{1}{\sqrt{2\pi |b|}} \int_\R T^A_sf(t)\, d\mu(s).
\end{equation}

The integral in (2.9) can also be viewed as a vector-valued integral as follows.
\begin{equation}\label{eq:translation vector valued}
\mu\star_Af=\frac{1}{\sqrt{2\pi |b|}}\int_\R T_s^Afd\mu(s),
\end{equation}
where the right hand side is a Bochner integral. The convergence of the integral follows from
\[
\int_\R \|T_s^Af\|_1d\mu(s)=\|f\|_1\mu(\R)<\infty.
\]

Now we give a relation between classical convolution and $A$-convolution of a measure and a function. Consider
\begin{align*}
(\mu \star_A f)(x)&=\frac{1}{\sqrt{2\pi |b|}}\int_\R T_s^Af(x)d\mu(s)\\
&=\frac{1}{\sqrt{2\pi |b|}}\int_\R e^{i\frac{a}{2b}s^2}C_\frac{a}{b}^{-1}T_sC_\frac{a}{b}f(x)d\mu(s)\\
&=\frac{e^{-i\frac{a}{2b}x^2}}{\sqrt{2\pi |b|}}\int_\R T_sC_\frac{a}{b}f(x)e^{i\frac{a}{2b}s^2}d\mu(s)\\
&=\frac{e^{-i\frac{a}{2b}x^2}}{\sqrt{2\pi |b|}}\int_\R T_sC_\frac{a}{b}f(x)d(C_\frac{a}{b}\mu)(s)\\
&=\frac{e^{-i\frac{a}{2b}x^2}}{\sqrt{|b|}}(C_\frac{a}{b}\mu \star C_\frac{a}{b}f)(x),
\end{align*}
using (\ref{chirpTranslation}), which in turn implies that 
\begin{equation}\label{eq:A-convolution}
C_\frac{a}{b}(\mu\star_Af)=\frac{1}{\sqrt{|b|}}(C_\frac{a}{b}\mu\star C_\frac{a}{b}f).
\end{equation}
Further, using \eqref{rhoT_x^A}, one can show that
\begin{equation}\label{eq:convolutions using rho}
\rho_A(\mu \star_A f)=\frac{1}{\sqrt{|b|}}(\rho_A\mu) \star (\rho_Af).
\end{equation}
The convolution theorem for the SAFT reads as follows. 
\begin{propn}\label{pro:SAFT2}
Let $\mu\in M(\R)$ and $f\in L^1(\R)$. Then 
\begin{equation}\label{eq:SAFT5}
\mathscr{F}_A(\mu\star_A f)(\omega) = \overline{\eta_A}(\omega)\mathscr{F}_A\mu(\omega)\, \mathscr{F}_Af(\omega).
\end{equation}

\begin{proof}
\textcolor{black}{The proof is straightforward using the operator $\rho_A$.}
\end{proof}
\end{propn}

\textcolor{black}{
In particular, if $\mu=g\, d t,$ for some $g\in L^1(\R)$ then 
\[\mathscr{F}_A(g\star_A f)(\omega) = \overline{\eta_A}(\omega)\mathscr{F}_Ag(\omega)\, \mathscr{F}_Af(\omega).\]
See [6] for more details.
}


\par \textcolor{black}{The concept of chirp modulation was used in \cite{Bhandari} to define a chirp convolution and to get sampling theorems. Although the $A$-translation is somehow included implicitly in the definition of the chirp convolution, it has not been used to its full extent in \cite{Bhandari} and hence the harmonic analysis of the special affine Fourier transform has not been developed. However we want to make use of our new translation from harmonic analysis point of view. Towards this end we first we prove an analogue of Wendel's theorem for the SAFT.}


\par We are now in a position to state one of our main results. The following statement is an analogue of Wendel's theorem in the context of the SAFT. 
\begin{theorem}[Wendel] \label{Wendel}
Let $T:L^1(\R)\to L^1(\R)$ be a bounded linear operator. Then the following statements are equivalent.
\begin{itemize}
\item[(i)] $TT_s^A=T_s^AT,$ for all $s\in \R$.
\item[(ii)] $T(f\star_A g)=Tf\star _A g=f\star_A Tg$, for all $f, g \in L^1(\R).$
\item[(iii)] There exists a unique $\mu \in M(\R)$ such that $Tf=\mu \star_A f$.
\item[(iv)] There exists a unique $\mu \in M(\R)$ such that $$\F_A(Tf)(\omega)=\overline{\eta_A}(\omega)\F_A\mu(\omega)\F_Af(\omega).$$
\item[(v)] There exists a unique $\phi \in L^\infty (\R)$ such that $\F_A(Tf)(\omega)=\phi(\omega)\F_Af(\omega).$
\end{itemize}
\end{theorem}
\begin{proof}
Let $E_{\rho_A}f(t)=\rho_A(t) f(t)$ and define $\tilde{T}:L^1(\R)\to L^1(\R)$ by $\tilde{T}=E_{\rho_A}\,T\,E_{\overline{\rho_A}}$.
Then using \eqref{rhoT_x^A} we get 
\begin{align*}
T\, T^A_x=T\,E_{\overline{\rho_A}}E_{\rho_A}T_x^A=&\rho_A(x)\,T\,E_{\overline{\rho_A}}\, T_x\, E_{\rho_A}=\rho_A(x)\, E_{\overline{\rho_A}}\, \tilde{T}\, T_x\, E_{\rho_A},
\end{align*}
which shows that $T\,T^A_x=T^A_x\, T$ iff $\tilde{T}\, T_x=T_x\, \tilde{T}$. Similarly we can show that $T(f\star_A g)=Tf\star_A g$ iff $\tilde{T}(f\star g)=\tilde{T}f\star g$, 
and $Tf=\mu\star_A\,f$ iff $\tilde{T}\,f=\frac{1}{\sqrt{|b|}} (E_{\rho_A}\mu)\star f$. The equivalence of (i),(ii), and (iii) now follows from Wendel's theorem in the classical case. That (iii) implies (iv) 
and (iv) implies (v) is obvious. To show that (i) follows from (v), let $f\in L^1(\R)$. For $s\in\R$ we have 
\begin{align*}
\mathscr{F}_A(T\,T^A_s\,f)(\omega)&=\phi(\omega)\,\mathscr{F}_A(T^A_s\, f)(\omega)\\[1ex]
&=\phi(\omega)\,\rho_A(\omega)\, e^{-i\,s\omega/b}\,\mathscr{F}_Af(\omega)\\[1ex]
&=\rho_A(\omega)\, e^{-i\,s\omega/b}\,\mathscr{F}_A(T\,f)(\omega)=\mathscr{F}_A(T^A_s\,T\,f)(\omega).
\end{align*}

\end{proof}

\par \textcolor{black}{We end this section by establishing an analogue of the Poisson summation formula for the SAFT and the corresponding $A$-translation. From now on, \textcolor{black}{we use the following notation. 
\[
I:=[-|b|\pi,|b|\pi].
\]}}
\begin{theorem}
Let $f\in L^1(\R)\cap L^2(\R)$. Then  the following formula holds. 
\[
\sqrt{2\pi |b|}\sum_{k\in \Z} T_{-2kb\pi}^Af(t)e^{-2iabk^2\pi^2}=
\sum_{n\in \Z}\overline{\eta_A}(n+p)\F_Af(n+p)e^{-\frac{i}{2b}(at^2-2nt)}, \quad t\in \R.
\]
\end{theorem}
\textcolor{black}{We refer to \cite{Bhandari} for the proof.}

\section{Closed ideals in $(L^1(\R),\star_A)$}
\textcolor{black}{We have seen that $$C_{\frac{a}{b}}(f\star_Ag)=\frac{1}{\sqrt{|b|}}(C_\frac{a}{b}f\star C_\frac{a}{b}g).$$ Thus it is easy to see that $\| g\star_A f\|_1\leq (2\pi|b|)^{-1/2} \|g\|_1\, \|f\|_1$, from which it follows that $(L^1(\R), \star_A)$ is a commutative Banach algebra. In this section, we aim to study the closed ideals in $(L^1(\R), \star_A)$. Towards this end, first we show that $\big(L^1(\R),\star_A\big)$ possesses a bounded approximate identity as in the classical case $(L^1(\R),\star)$.}

\begin{theorem}\label{theo:SAFT1}
The space $(L^1(\R),\star_A)$ possesses a bounded approximate identity.
\end{theorem}
\begin{proof}
Let $\{g_\alpha\}$ be a bounded approximate identity in $(L^1(\R), \star )$ and define $u_\alpha=\sqrt{|b|}\, \overline{\rho_A}\, g_\alpha$. Then using 
\eqref{eq:convolutions using rho}, we get 
\begin{align*}
\|f\star_A u_\alpha-f\|_1=\|\rho_A(f\star_Au_\alpha)-\rho_Af\|_1&=\Big\|{\textstyle\frac{1}{\sqrt{|b|}}}\big((\rho_Af)\star(\rho_A u_\alpha)\big)-f\Big\|_1\\&=\|(\rho_Af)\star g_\alpha-\rho_Af\|_1\to 0
\end{align*}
as $\alpha\to\infty$. Hence $\{u_\alpha\}$ is a bounded approximate identity for $(L^1(\R),\star_A)$.

\end{proof}

\textcolor{black}{Now, we aim to study $A$-translation invariant closed ideals in $(L^1(\R),\star_A)$. First we prove the following theorem in this context.}

\textcolor{black}{\begin{theorem}\label{th: ideal}
Let $J$ be a closed subspace of $L^1(\R)$. Then $J$ is an ideal in $(L^1(\R),\star_A)$ if and only if it is invariant under $A$-translations. 
\end{theorem}
\begin{proof}
Let $J$ be an ideal in $(L^1(\R),\star_A)$. Let $\{u_\alpha\}$ be an approximate identity in $(L^1(\R),\star_A)$. Then for $f\in J$ and $x\in \R$, we have
\[
T_x^Af=\lim_{\alpha\to\infty}T_x^A(u_\alpha\star_A f)=\lim_{\alpha\to\infty}T_x^Au_\alpha \star_A f,
\]
using Theorem \ref{Wendel}. Since $T_x^Au_\alpha\star_Af\in J$, for all $\alpha$ and $J$ is closed, $T_x^Af\in J$.\\
Conversely, assume that $J$ is invariant under $A$-translations. Let $f\in J,~g\in L^1(\R)$. Then viewing $f\star_Ag$ as a Bochner integral as in \eqref{eq:translation vector valued}, we can conclude that $f\star_Ag\in J$. 
\end{proof}}

\textcolor{black}{\begin{propn}\label{pro:cpt support dense}
The collection $\{g\in L^1(\R):\F_A(g) \textit{ has compact support}\}$ is dense in $L^1(\R)$.
\end{propn}
\begin{proof}
We know that $\{g\in L^1(\R):\widehat{g} \textit{ has compact support}\}$ is dense in $L^1(\R)$. Thus, for $f\in L^1(\R)$, there exists $g\in L^1(\R)$ such that $\widehat{g}$ has compact support and 
  $\|C_\frac{a}{b}f-g\|<\epsilon$.
  This implies that \[\|f-C_{-\frac{a}{b}}g\|=\|C_\frac{a}{b}f-C_\frac{a}{b}C_{-\frac{a}{b}}g\|<\epsilon.\]
  Further, $\F_A(C_{-\frac{a}{b}}g)$ has compact support as $\widehat{g}$ has compact support and
    \[
    \F_A(C_{-\frac{a}{b}}g)(\omega)=\frac{\eta_A(\omega)}{\sqrt{|b|}}\widehat{g}(\frac{\omega-p}{b}), 
    \]
    which completes the proof.
\end{proof}}

\textcolor{black}{\begin{lemma}[Lemma 4.59 in \cite{Folland1995}]\label{lemma:weiner lemma 1}
Let $f\in L^1(\R)$ and $\omega_0\in \R$. Then for every $\epsilon>0$, there exists $h\in L^1(\R)$ with $\|h\|_1<\epsilon$ such that
\[
(f+h)~\widehat{}~(\omega)=\widehat{f}(\omega_0),
\]
for every $\omega$ in some neighbourhood of $\omega_0$.
\end{lemma}}

\textcolor{black}{Now we are in a position to state and prove an analogue Weiner's theorem in connection with the SAFT.}

\textcolor{black}{\begin{theorem}
    Let $J$ be a closed $A$-translation invariant subspace of $L^1(\R)$ such that $Z(J)=\emptyset$, where $Z(J):=\{\omega\in\R:\F_Af(\omega)=0,~ \textit{ for all }f\in J  \}$. Then $J=L^1(\R)$.
\end{theorem}
\begin{proof}
  In view of Proposition \ref{pro:cpt support dense}, it is enough to show that $f\in J$ for all $f\in L^1(\R)$ such that $\F_Af$ has compact support. Let $f\in L^1(\R)$ be such that $\F_Af$ has compact support. Let $K=supp(\F_Af)$.\\
Step 1: In this step we show that for each $\omega_0\in\R$, there exists $F\in J$ such that $\F_Af (\omega)=\F_AF(\omega)$ in a neighbourhood of $\omega_0$.
Since $Z(J)=\emptyset$, we can choose $g\in J$ such that $\overline{\eta_A}(\omega_0)\F_Ag(\omega_0)=1$. Then using Lemma \ref{lemma:weiner lemma 1}, there exists $h\in L^1(\R)$ with $\|h\|_1<\frac{(2\pi |b|)^{1/2}}{2}$ and 
\[
(C_\frac{a}{b}g)~\widehat{}~(\frac{\omega-p}{b})+(C_\frac{a}{b}h)~\widehat{}~(\frac{\omega-p}{b})=(C_\frac{a}{b}g)~\widehat{}~(\frac{\omega_0-p}{b}),
\]
in a neighbourhood of $\omega_0$. This implies that
\begin{equation}\label{eq:weiner 1}
\F_Ag(\omega)+\F_Ah(\omega)=\frac{\eta_A(\omega)}{\eta_A(\omega_0)}\F_Ag(\omega_0)=\eta_A(\omega),
\end{equation}
in a neighbourhood of $\omega_0$. Let $h_n=h\star_A h\star_A\cdots \star_A h$ ($n$-times). Then using the fact that $(L^1(\R),\star_A)$ is a Banach algebra, we can show that the series $f+\sum_{n\in \N}f\star_Ah_n$ converges in $L^1(\R)$. Let $k=f+\sum_{n\in \N}f\star_Ah_n$. Then using convolution theorem, we obtain
\begin{align*}
    \F_Ak(\omega)&=\F_Af(\omega)+\sum_{n\in \N} \overline{\eta_A}(\omega)^n\F_Af(\omega)\big(\F_Ah(\omega)\big)^n\\&=\F_Af(\omega)\frac{1}{1-\overline{\eta_A}(\omega)\F_Ah(\omega)}=\frac{\F_Af(\omega)}{\overline{\eta_A}(\omega)\F_Ag(\omega)},
\end{align*}
in a neighborhood of $\omega_0$, using \eqref{eq:weiner 1}. The second equality in the above equation follows from the fact that $|\F_Ah(\omega)|\leq \frac{1}{(2\pi |b|)^{1/2}}\|h\|_1<\frac{1}{2}$. Thus 
\[
\F_Af(\omega)=\overline{\eta_A}(\omega)\F_Ag(\omega)\F_Ak(\omega)=\F_A(g\star_A k)(\omega),
\]
in a neighbourhood of $\omega_0$. As $g\star_A k\in J$, by Theorem \ref{th: ideal}, our claim is established.\\
Step 2: In this step we show that $f\in J$. Appealing to Step 1, for each $\omega\in\R$, choose $g_\omega\in J$ such that $\F_Af=\F_Ag_\omega$ on a neighborhood $U_\omega$ of $\omega$. Using compactness of $K$, we get $U_1,\, U_2,\, \cdots U_n$ and $g_1, g_2, \cdots g_n\in J$ such that $K\subset \cup_{i=1}^nU_i$ and $\F_Af=\F_Ag_i$ on $U_i$. Now choose open $W_\omega$ such that
\[
\{\omega\}\subset W_\omega\subset \overline{W_\omega}\subset U_i.
\]
Again using the compactness of $K$, there exist $W_{\omega_1},\, W_{\omega_2},\cdots, W_{\omega_m}$ such that $K\subset\cup_{i=1}^m\overline{W_{\omega_i}}$ and $\overline{W_{\omega_i}}\subset U_{\omega_i}$, where $\omega_i\in U_{\omega_i}\in \{U_1,\,U_2,\cdots, U_n\}$. Take $h_1,\,h_2,\cdots, h_m\in L^1(\R)$ such that
\[
\overline{\eta_A}(\omega)\F_Ah_j(\omega)=1~~\text{ on } \overline{W_{\omega_j}}~~\text{ and } supp(\F_Ah_j)\subset U_{\omega_j}.
\]
Then $\Pi_{j=1}^m(1-\overline{\eta_A}(\omega)\F_Ah_j(\omega))=0$ on $K$. This implies that
\[
\F_Af(\omega)=\F_Af(\omega)[1-\Pi_{j=1}^m(1-\overline{\eta_A}(\omega)\F_Ah_j(\omega))].
\]
This can be rewritten as $f=\sum f\star_AH_i$, where $H_i$ being one of the $h_j$'s or their convolutions, and $supp(\F_AH_i)\subset U_{\omega_i}$, for some $i$. But
\begin{align*}
\F_A(f\star_AH_i)(\omega)=\overline{\eta_A}(\omega)\F_Af(\omega)\F_AH_i(\omega)&=\overline{\eta_A}(\omega)\F_Ag_i(\omega)\F_AH_i(\omega)\\
&=\F_A(g_i\star_AH_i)(\omega).
\end{align*}
As $g_i\star_AH_i\in J$, $f\in J$. This completes the proof.
\end{proof}}

\textcolor{black}{\begin{corollary}
Let $f\in L^1(\R)$. Then the closed linear span of $A$-translates of $f$ is $L^1(\R)$ if and only if $\F_Af$ never vanishes.
\end{corollary}
\begin{proof}
Let $J$ be the closed linear span of $A$-translates of $f$.
    Let $\F_Af(\omega)\neq 0$ for all $\omega\in \R$. Since $J$ is a closed $A$-translation invariant subspace, appealing to Wiener's theorem, we get $J=L^1(\R)$.\\
    Conversely assume that $J=L^1(\R)$. Suppose $\F_Af(\omega_0)=0$ for some $\omega_0\in \R$. Then $\F_Ah(\omega_0)=0$ for all $h\in span\{T_x^Af:x\in \R\}$. Since $span\{T_x^Af:x\in \R\}$ is dense in $L^1(\R)$ and $\|\F_Ah\|_\infty\leq (2\pi |b|)^{-1/2}\|f\|_1$, we can conclude that $\F_Ah(\omega_0)=0$, for all $h\in L^1(\R)$, which is an impossibility. Hence the result follows.
\end{proof}}

Now, we shall state the analogue of Wiener's theorem for $L^2(\R)$ functions.

\textcolor{black}{\begin{theorem}
    Let $f\in L^2(\R)$. Then the closed linear span of $A$-translates of $f$ is $L^2(\R)$ if and only if $\F_Af(\omega)\neq 0$ a.e.
\end{theorem}
\begin{proof}
    Let $M=\overline{span}\{T_x^Af:x\in \R\}$. Then $g\in M^\perp$ if and only if $\langle T_x^Af,g\rangle=0$, for all $x\in\R$. For $x\in \R$, consider
    \begin{align*}
        \langle T_x^Af,g\rangle=\langle \F_A(T_x^Af),\F_Ag\rangle&=\int_\R \rho_A(x)e^{-i\frac{x}{b}\omega}\F_Af(\omega)\overline{\F_Ag(\omega)}d\omega\\
        &=\sqrt{2\pi}\rho_A(x)(\F_Af\overline{\F_Ag})~\widehat{}~(\frac{x}{b}).
    \end{align*}
    Thus \textcolor{black}{$\langle T_x^Af,g\rangle =0~\text{ for all }x\in \R$ is equivalent to 
    \[(\F_Af(\omega)\overline{\F_Ag(\omega)})~\widehat{ }~(x)=0,~\text{ for all }x\in \R,\]
    which is same as $\F_Af(\omega)\overline{\F_Ag(\omega)}=0,\text{ a.e. }\omega\in \R.$} 
    This shows that $M^\perp=\{0\}$ if and only  if $\F_Af(\omega)\neq 0$ a.e. $\omega\in\R.$
\end{proof}}

\textcolor{black}{We conclude this section by establishing an analogue of Wiener-Tauberian theorem.} \textcolor{black}{Recall that $M_y$ denotes the modulation operator which is defined in \eqref{eq:cmodulation}.}

\textcolor{black}{\begin{theorem}[Wiener-Tauberian]
    Let $\phi\in L^\infty(\R),$ $h\in L^1(\R)$ be such that $\F_Ah(\omega)\neq 0$, for all $\omega\in \R$ and 
    \[
    \lim_{x\to \pm \infty}C_\frac{a}{b}(h\star_A \phi)(x)=m\F_A(M_{-\frac{p}{b}}h)(0).
    \]
    Then
    \[
    \lim_{x\to \pm \infty}C_\frac{a}{b}(f\star_A \phi)(x)=m\F_A(M_{-\frac{p}{b}}f)(0),
    \]
    for all $f\in L^1(\R)$.
\end{theorem}
\begin{proof}
     Since, $0\neq \F_Ah(\omega)=\frac{\eta_A(\omega)}{\sqrt{|b|}}(C_\frac{a}{b}h)~\widehat{}~(\frac{\omega-p}{b})$, for all $\omega\in \R$, $(C_\frac{a}{b}h)~\widehat{}~(\omega)\neq 0$, for all $\omega\in \R$. Further,
     \[
     C_\frac{a}{b}(h\star_A \phi)(x)=\frac{1}{\sqrt{|b|}}(C_\frac{a}{b}h\star C_\frac{a}{b}\phi)(x).
     \]
     Furthermore, 
     \begin{align*}
     \F_Af(0)=\frac{1}{\sqrt{|b|}}(C_\frac{a}{b}f)~\widehat{}~(-\frac{p}{b})=\frac{1}{\sqrt{|b|}}T_{\frac{p}{b}}(C_\frac{a}{b}f)~\widehat{}~(0)&=\frac{1}{\sqrt{|b|}}(M_\frac{p}{b}C_\frac{a}{b}f)~\widehat{}~(0)\\
     &=\frac{1}{\sqrt{|b|}}(C_\frac{a}{b}M_\frac{p}{b}f)~\widehat{}~(0),
     \end{align*}
     using $C_sM_y=M_yC_s$. Thus
     \begin{equation}\label{eq:wt 1}
     \F_A(M_{-\frac{p}{b}f})(0)=\frac{1}{\sqrt{|b|}}(C_\frac{a}{b}f)(0).
     \end{equation}
     Hence 
     \[
     \lim_{x\to\pm\infty}(C_\frac{a}{b}h\star C_\frac{a}{b}\phi)(x)=m(C_\frac{a}{b}h)~\widehat{}~(0),
     \]
     using given hypothesis. Now using the classical Wiener-Tauberian theorem, we get
     \[
     \lim_{x\to\pm\infty}(f\star C_\frac{a}{b}\phi)(x)=m\widehat{f}(0),~~\text{ for all }f\in L^1(\R).
     \]
     This implies that
     \[
     \lim_{x\to\pm\infty}(C_\frac{a}{b}f\star C_\frac{a}{b}\phi)(x)=m(C_\frac{a}{b}f)~\widehat{}~(0),~~\text{ for all }f\in L^1(\R).
     \]
     This in turn implies that
     \[
    \lim_{x\to \pm \infty}C_\frac{a}{b}(f\star_A \phi)(x)=m\F_A(M_{-\frac{p}{b}}f)(0),~~\text{ for all }f\in L^1(\R),
    \]
    using \eqref{eq:wt 1}.
\end{proof}
As a special case, we state the following analogue of the Wiener-Tauberian theorem associated with the fractional Fourier transform.
\begin{corollary}
    Let $\phi\in L^\infty(\R)$, $h\in L^1(\R)$ be such that $\F_\theta h(\omega)\neq 0$, for all $\omega\in \R$ and 
    \[
    \lim_{x\to \pm \infty}C_\theta(h\star_\theta \phi)(x)=m\F_\theta h(0).
    \]
    Then
    \[
    \lim_{x\to \pm \infty}C_\theta(f\star_\theta \phi)(x)=m\F_\theta f(0),
    \]
    for all $f\in L^1(\R)$.
\end{corollary}}

\section{$A$-shift invariant spaces}

\textcolor{black}{In this section, we aim to study shift invariant spaces associated with the SAFT, called $A$-shift invariant spaces, in detail. Recall that the $A$-shift invariant space is defined by $V_A(\phi)=\overline{span}\{T_k^A\phi:k\in\Z\}$ for $\phi\in L^2(\R)$.} First, we obtain the following result whose proof is similar to that of Theorem 2 in \cite{AldGro2000} in the classical case. 
\begin{theorem}
Let $\phi \in W\big(C, \ell^1(\Z)\big)$ be such that $\{T_k^A \phi : k\in \Z\}$ forms a Riesz basis for $V_A(\phi)$. Then
\begin{itemize}
\item[(i)] $V_A(\phi)\subseteq W\big(C, \ell^2(\Z)\big)$.
\item[(ii)] If $\mathcal{X}= \{x_k:k\in \Z\}$ is separated, then there is $M>0$ such that
$$
\big(\sum_{k \in \Z} |f(x_k)|^2\big)^{1/2}\leq M\|f\|, \quad \forall f\in V_A(\phi).
$$
\end{itemize}

\end{theorem}

\begin{corollary}
Let $\phi \in W(C, \ell^1(\Z))$. Then $V_A(\phi)$ is a reproducing kernel Hilbert space with the reproducing kernel
\[
K(x,y)=\sum_{k\in \Z}e^{i\frac{a}{b}k(x-y)}\overline{\phi(x-k)}S^{-1}\phi (y-k),
\]
where $S$ is the frame operator for $\{T_k^A\phi:k\in\Z\}$.
\begin{proof}
Let $x\in \R$ be fixed. Then, taking $\mathcal{X}=\{x+k:k\in \Z\}$ in the previous theorem we get $M>0$ such that $|f(x)|\leq M\|f\|$, for every $f\in V_A(\phi)$, which shows that $V_A(\phi)$ is a RKHS. The reproducing kernel for $V_A(\phi)$ is $K(x,y)=\sum_{k\in \Z}\overline{T_k^A\phi (x)}T_k^AS^{-1}\phi (y)$. Now, using the definition of $T_k^A$ our assertion follows.
\end{proof}
\end{corollary}

\begin{theorem}
If $\{T_k^A\phi:k\in \Z\}$ is a frame sequence, for $\phi \in L^2(\R)$, then the members of its canonical dual frame also are $A$-translates of a single function.
\begin{proof}
Let $S$ be the frame operator for $\{T_k^A\phi:k\in\Z\}$. First we prove that $ST_k^A=T_k^AS,~\forall ~k \in \Z.$ Let $f\in V_A(\phi)$. Then for $k\in \Z$, we have
\begin{align*}
ST_k^Af&= \sum_{k
^\prime \in \Z} \langle T_k^Af, T_{k^\prime}^A \phi \rangle T_{k^\prime}^A \phi\\
&=\sum_{k^ \prime \in \Z} \langle f, (T_k^A)^*T_{k^\prime}^A\phi \rangle T_{k^ \prime}^A\phi\\
&=\sum_{k^\prime \in \Z}\langle f, e^{-i\frac{a}{b}k^2}T_{-k}^AT_{k^\prime}^A\phi \rangle T_{k^\prime}^A\phi\\
&=\sum_{k^\prime \in \Z}\langle f, T_{k^\prime -k}^A\phi \rangle e^{-i\frac{a}{b}k(k^\prime -k)}T_{k^\prime}^A\phi\\
&= \sum_{\ell \in \Z} \langle f, T_\ell^A \phi \rangle T_k^AT_\ell^A \phi\\
&=T_k^ASf.
\end{align*}
Since $S$ is invertible, $T_k^Af=S^{-1}T_k^ASf$ and we have for every $h \in V_A(\phi),$
$$
T_k^AS^{-1}h=S^{-1}T_k^ASS^{-1}h=
S^{-1}T_k^Ah.
$$
Thus, if $\{T_k^A\phi:k\in \Z\}$ is a frame for $V_A(\phi),$ then the canonical dual frame $\{S^{-1}T_k^A\phi :k\in \Z\}$ is given by $\{T_k^AS^{-1}\phi : k\in \Z\}.$ Taking $\psi =S^{-1}\phi$, we conclude that the canonical dual frame is also of the form $\{T_k^A\psi:k\in \Z\}.$
\end{proof}
\end{theorem}

\begin{remark}
If we assume $\{T_k^A\phi:k\in \Z\}$ is a Riesz basis for $V_A(\phi)$, then $\{T_k^AS^{-1}\phi : k\in \Z\}$ is the dual Riesz basis of $\{T_k^A\phi:k\in \Z\}$. 
\end{remark}


\textcolor{black}{Now we obtain a characterization for the system of $A$-translates $\{T_k^A\phi:k\in \Z\}$ to be a frame sequence in terms of the weight function $w_\phi(\omega):=\sum_{k\in \Z}|\F_A\phi(\omega+2kb\pi)|^2$.}

\textcolor{black}{\begin{theorem}\label{th: frame char}
    Let $\phi\in L^2(\R)$. Then the system of $A$-translates $\{T_k^A\phi:k\in \Z\}$ is a frame sequence with bounds $m,M>0$ if and only if
    \begin{equation}\label{eq: frame char 2}
    \frac{m}{2\pi |b|}\chi_{E_\phi}\leq \sum_{k\in \Z}|\F_A\phi(\omega+2kb\pi)|^2\leq \frac{M}{2\pi |b|}\chi_{E_\phi},\, ~~\omega\in I,
    \end{equation}
    where $E_\phi=\{\omega\in\R:w_\phi(\omega)\neq 0\}$ and $I=[-|b|\pi,|b|\pi]$.
\end{theorem}
\begin{proof}
    Let $\{T_k^A\phi:k\in \Z\}$ be a frame sequence with bounds $m,M>0$. Then
    \begin{equation}\label{eq:frame char 1}
    m\|f\|^2\leq \sum_{k\in \Z}|\langle f,T_k^A\phi\rangle|^2\leq M\|f\|^2, ~~\text{ for all }f\in V_A(\phi).
    \end{equation}
    Let $F$ be a finite subset of $\Z$. Let $f=\sum_{k\in F}c_kT_k^A\phi\in V_A(\phi)$. Then $\F_Af (\omega)=r(\omega)\F_A \phi(\omega)$, where $r(\omega)=\sum_{k\in F}c_k\rho_A(k)e^{-\frac{i}{b}k\omega}$. Thus
    \begin{align}\label{eq:frame char 3}
        \|f\|^2=\langle f,f\rangle = \langle \F_Af, \F_Af\rangle=\int_\R |r(\omega)|^2|\F_A\phi|^2d\omega=\int_I |r(\omega)|^2w_\phi(\omega)d\omega.
    \end{align}
    Similarly,
    \begin{align*}
        \langle f,T_k^A\phi\rangle=\int_\R \F_Af(\omega)\overline{\F_A(T_k^A\phi)(\omega)}d\omega &=\int_\R r(\omega)|\F_A\phi(\omega)|^2\overline{\rho_A(k)}e^{\frac{i}{b}k\omega}d\omega\\
        &=\int_I r(\omega)\overline{\rho_A(k)}w_\phi(\omega)e^{\frac{i}{b}k\omega}d\omega\\
        &=\sqrt{2\pi |b|}\big(r(\omega)w_\phi(\omega)\big)~\widehat{}~(k),
    \end{align*}
    where $\widehat{f}(k)=\int_I f(x)\overline{\rho_A(k)}e^{\frac{i}{b}kx}dx$. Hence
    \begin{equation}\label{eq: frame char 4}
    \sum_{k\in \Z}|\langle f, T_k^A\phi\rangle|^2=2\pi |b|\sum_{k\in \Z}|\big(r(\omega)w_\phi(\omega)\big)~\widehat{}~(k)|^2=2\pi |b|\int_I |r(\omega)w_\phi(\omega)|^2d\omega.
    \end{equation}
    Now, using \eqref{eq:frame char 1}, we get
    \[
    m\int_I |r(\omega)|^2w_\phi(\omega)d\omega\leq 2\pi |b|\int_I |r(\omega)|^2\big(w_\phi(\omega)\big)^2d\omega\leq M\int_I |r(\omega)|^2w_\phi(\omega)d\omega,
    \]
    for every $|b|$-periodic trigonometric polynomial $r$. This implies that
    \[
    \frac{m}{2\pi |b|}w_\phi(\omega)\leq w_\phi(\omega)^2\leq \frac{M}{2\pi |b|}w_\phi(\omega),
    \]
    a. e. $\omega\in I$, from which \eqref{eq: frame char 2} follows.\\
    Conversely, assume that \eqref{eq: frame char 2} holds. Then using \eqref{eq:frame char 3}, \eqref{eq: frame char 4}, we can show that \eqref{eq:frame char 1} holds for all $f\in span\{T_k^A\phi:k\in \Z\}$. Since $span\{T_k^A\phi:k\in \Z\}$ is dense in $V_A(\phi)$, the proof follows.
\end{proof}}


\par \textcolor{black}{
In a similar way, we can obtain the characterizations for the system $\{T_k^A\phi:k\in \Z\}$ to be a Riesz sequence or an orthonormal system. We state the results without proof. The interested readers can see the proof form \cite{harnandez_weiss} and from \cite{Bhandari}.}

\begin{theorem}
Let $\phi \in L^2(\R)$. Then the collection $\{T_k^A\phi:k\in \Z\}$ is a Riesz basis for $V_A(\phi)$ if and only if there are $m,M>0$ such that 
\[
m\leq \sum_{k\in \Z} |\F_A\phi (\omega +2k b\pi)|^2\leq M,
\]
for almost all $\omega \in I.$
\end{theorem}
\begin{theorem}\label{th:ons char}
Let $\phi \in L^2(\R)$. Then the collection $\{T_k^A\phi : k \in \Z\}$ is an orthonormal system in $L^2(\R)$ if and only if
\begin{equation}\label{eq:ons}
\sum_{k\in \Z} |\F_A\phi(\omega + 2kb\pi)|^2=\frac{1}{2\pi |b|},
\end{equation}
for almost all $\omega \in I.$
%
\end{theorem}

\section{Sampling in $A$-shift invariant spaces}

In order to get an equivalent condition for the stable set of sampling in terms of the Zak transform, we first introduce $A$-Zak transform.
\begin{definition}
The $A$-Zak transform $\Za_Af $ of a function $f\in L^2(\R)$ is a function on $\R ^2$, defined as
$$
\Za _Af(t, \omega)=\frac{\overline{\eta}_A(\omega)}{\sqrt{2\pi |b|}}\sum_{k\in \Z} T_k^Af(t)e^{-\frac{i}{2b}(ak^2-2k\omega +2pk)},~t, \omega \in \R.
$$ 
One can simplify the right hand side and get 
$$
\Za_Af(t,\omega)=\frac{\overline{\eta_A}(\omega)}{\sqrt{2\pi |b|}}\sum_{k\in \Z} f(t-k)e^{\frac{i}{2b}(ak^2-2akt+2k\omega -2pk)},\quad \text{for}~ t, \omega \in \R.
$$
\end{definition}
\begin{remark}
In particular, if we take $A=\{0,1,-1,0,0,0\}$, then $A$-Zak transform reduces to the classical Zak transform.
\end{remark}

\begin{theorem}
The $A$-Zak transform is an isometry between the spaces $L^2(\R)$ and $L^2\big([0,1]\times I\big)$.
\end{theorem}

See \cite{Bhandari} for the proof.\\

Define an operator $T:L^2(I)\to V_A(\phi) $ by 
$$
TF(x)=\sum_{k\in \Z} \langle F,E_k\rangle T_{-k}^A\phi (x), \quad F\in L^2(I),
$$
where $E_k(t)=\frac{\eta_A(t)}{\sqrt{2\pi|b|}}e^{\frac{i}{2b}(ak^2-2pk+2kt)}$. Clearly $\{E_k:k\in \Z\}$ is an orthonormal basis for $L^2(I).$
\par Suppose $\{T_k^A\phi:k\in \Z\}$ is a Riesz sequence. Then there are constants $m,M>0$ such that
$$
m\big(\sum_{k \in \Z} |\langle F, E_k \rangle |^2\big)^{1/2}\leq \|\sum_{k \in \Z} \langle F, E_k \rangle T_{-k}^A\phi\|_{L^2(\R)}\leq M\big(\sum_{k \in \Z}|\langle F, E_k\rangle |^2\big)^{1/2},
$$
for all $F\in L^2(I)$. Since $\{E_k:k\in \Z\}$ is an orthonormal basis for $L^2(I)$ and $F\in L^2(I)$, the above inequality reduces to 
$$
m\|F\|_{L^2(I)}\leq \|TF\|_{L^2(\R)}\leq M\|F\|_{L^2(I)}, \quad \forall~ F\in L^2(I).
$$
This shows that $T$ is bounded above and bounded below. By Riesz-Fischer theorem $T$ is onto. Hence $T$ is invertible. Moreover, we have
\begin{align} 
\nonumber TF(x)&=  \sum_{k \in \Z} \langle F, E_k \rangle T_{-k}^A \phi (x)=\big\langle F, \frac{\eta_A(\cdot)}{\sqrt{2\pi |b|}}\sum_{k \in \Z} e^{\frac{i}{2b}(ak^2-2pk+2k\cdot)}\overline{T_{-k}^A\phi(x)} \big\rangle\\ 
&= \langle F, \overline{\Za_A\phi(x, \cdot)}\rangle. \label{TZak}
\end{align}
\par Now, we are in a position to prove equivalent conditions for stable set of sampling for a shift invariant space $ V_A(\phi).$

\begin{theorem} \label{Seq}
Assume that $V_A(\phi)$ is a reproducing kernel Hilbert space, for $\phi \in L^2(\R)$, such that $\{T_k^A\phi:k\in \Z\}$ forms a Riesz basis for $V_A(\phi)$. Then the following statements are equivalent.
\begin{itemize}
\item[(i)] The set $\mathcal{X}=\{x_j:j\in \Z\}$ is a stable set of sampling for $V_A(\phi).$
\item[(ii)] There are constants $m,M>0$ such that
$$
m\|d\|_{\ell^2(\Z)}^2\leq \|Ud\|_{\ell^2(\Z)}^2\leq M\|d\|_{\ell^2(\Z)}^2, \quad \forall~ d\in \ell^2(\Z),
$$
where the operator $U=(U_{j,k})$ is defined by $$U_{j,k}=T_k^A\phi(x_j)=e^{-i\frac{a}{b}k(x_j-k)}\phi(x_j-k).$$
\item[(iii)] The set of reproducing kernels $\{K_{x_j}:j\in \Z\}$ for $V_A(\phi)$ is a frame for $V_A(\phi)$.
\item[(iv)] The set $\{\Za_A\phi(x_j, \cdot):j\in \Z\}$ is a frame for $L^2(I)$.
\end{itemize}
\begin{proof}
The proof is similar to the classical case. However for the sake of completion, we give the outline of the proof. (i) $\Leftrightarrow$ (ii): If $f\in V_A(\phi)$ then there is $d=\{d_k\}\in \ell^2(\Z)$ such that $f(x)= \sum_{k \in \Z}d_kT_k^A\phi(x)$, and hence $f(x_j)=\sum_{k \in \Z} d_k T_k^A\phi (x_j)=(Ud)_j$. Since $\{T_k^A\phi : k\in \Z\}$ is a Riesz basis for $V_A(\phi)$, the following statements are equivalent.
\begin{itemize}
\item[(a)] There are constants $m,M>0$ such that for every $f\in V_A(\phi)$
$$m\|f\|^2\leq \sum_{j \in \Z} |f(x_j)|^2\leq M\|f\|^2.$$
\item[(b)] There are constants $m^\prime , M^\prime >0$ such that for $d\in \ell^2(\Z)$
$$
m^\prime \|d\|^2 \leq \|Ud\|^2\leq M^\prime \|d\|^2.
$$
\end{itemize} 
In fact, if (a) holds, then 
$$
m\|f\|^2\leq \|Ud\|^2\leq M \|f\|^2.
$$
Now by using the fact that $\{T_k^A\phi :k\in \Z\}$ is a Riesz sequence, one can find $m^{\prime \prime}, M^{\prime \prime}>0$ such that $m^{\prime \prime}\|d\|^2\leq\|f\|^2\leq M^{\prime \prime} \|d\|^2 $. Thus (b) follows from (a). Similarly one can prove (b) implies (a).\\
(i)$\Leftrightarrow $ (iii): Let $\{K_{x_j}:j\in \Z\}$ be the set of reproducing kernels for $V_A(\phi)$. Then the equivalence follows from the identity $f(x_j)=\langle f, K_{x_j}\rangle$.\\
(iii)$\Leftrightarrow $ (iv): Using (\ref{TZak}) we obtain 
$$
\sum_{j\in \Z}|f(x_j)|^2=\sum_{j\in \Z} |\langle f, K_{x_j}\rangle |^2=\sum_{j\in \Z} |\langle F, \overline{\Za_A\phi (x_j,\cdot)}\rangle |^2,
$$
here $TF=f$. Since $T$ is invertible, we get the equivalence of (iii) and (iv).
\end{proof}
\end{theorem}

\par Let $\phi \in W(C, \ell^1(\Z))$. Then we define the function $\phi_A^\dagger $ on the interval $I$, by
$$
\phi_A^\dagger (\omega)=\sum_{n\in \Z} \phi(n)\rho_A(n)e^{-\frac{i}{b}n\omega}.
$$
From the definition of the $A$-Zak transform we obtain $\Za_A\phi(0,\omega)=\frac{\overline{\eta_A}(\omega)}{\sqrt{2\pi|b|}}\phi_A^\dagger (\omega)$.
\begin{theorem} \label{Zss}
Let $\phi \in W\big(C,\ell^1(\Z)\big)$. Then the operator $U:\ell^2(\Z)\to \ell^2(\Z)$ defined by $U_{j,k}=T_k^A\phi(j)$ satisfies the inequalities
$$
\|\phi_A^\dagger\|_0^2~\|d\|^2\leq \|Ud\|^2\leq \|\phi_A^\dagger \|_\infty ^2~\|d\|^2, \quad \forall~ d\in \ell^2(\Z),
$$
where $\|\phi_A^\dagger \|_0=\inf_{x\in I}|\phi_A^\dagger (x)|,~\|\phi_A^\dagger \|_\infty =\sup_{x\in I}|\phi_A^\dagger (x)|.$
\begin{proof}
Let $d=\{d_n\}\in \ell^2(\Z)$. Then 
$$
(Ud)_n = \sum_{m\in \Z} U_{n,m}d_m = \sum_{m\in \Z} e^{-\frac{i}{b}am(n-m)}\phi (n-m)d_m.
$$
Since $\{\frac{1}{\sqrt{2\pi |b|}}\rho_A(n)e^{-\frac{i}{b}n\omega} :n\in \Z\}$ is an orthonormal basis for $L^2(I)$, we have
\begin{align*}
2\pi |b|~ \|Ud\|^2&= \int_I |\sum_{n\in \Z} (Ud)_n\rho_A(n)e^{-\frac{i}{b}n\omega} |^2 d\omega\\
&= \int_I|\sum_{n\in \Z} \sum_{m\in \Z} e^{-\frac{i}{b}am(n-m)}\phi(n-m)d_m\rho_A(n)e^{-\frac{i}{b}n\omega}|^2d\omega \\
&= \int_I |\sum_{n\in \Z} \sum_{m\in \Z} \rho_A(n-m)\phi(n-m)d_m\rho_A(m)e^{-\frac{i}{b}n\omega}|^2d\omega \\
&=\int_I |\sum_{n\in \Z} \sum_{m\in \Z} \rho_A(n)\phi(n)d_m\rho_A(m)e^{-\frac{i}{b}(m+n)\omega}|^2d\omega \\
&= \int_I |\phi_A^\dagger (\omega)|^2 |\sum_{m\in \Z}d_m\rho_A(m)e^{-\frac{i}{b}m\omega}|^2d\omega.
\end{align*}
This implies 
\begin{align*}
\frac{\|\phi_A^\dagger \|_0^2}{2\pi |b|}\int_I |\sum_{m\in \Z}d_m\rho_A(m)e^{-\frac{i}{b}m\omega }|^2d\omega &\leq \|Ud\|^2 \\ &\leq \frac{\|\phi_A^\dagger \|^2_\infty}{2\pi |b|}\int_I |\sum_{m\in \Z} d_m \rho_A(m)e^{-\frac{i}{b}m\omega}|^2d\omega
\end{align*}

or equivalently $\|\phi_A^\dagger \|_0^2\sum_{m\in \Z}|d_m|^2\leq \|Ud\|^2\leq \|\phi_A^\dagger \|_\infty ^2\sum_{m\in \Z} |d_m|^2$, from which the result follows.
\end{proof}
\end{theorem}

As a consequence we obtain the following
\begin{corollary}
Let $\phi \in W(C, \ell^1(\Z))$ be such that $\{T_k^A\phi :k\in\Z\}$ forms a Riesz basis for $V_A(\phi)$ and $\phi_A^\dagger (x)\neq 0$ for all $x\in I$. Then $\Z$ is a stable set of sampling for $V_A(\phi).$
\begin{proof}
Since $\phi_A^\dagger(x)\neq 0$ for all $x\in I$, it follows from Theorem \ref{Zss} that $U$ is bounded above and below. Then the assertion follows from Theorem \ref{Seq}.
\end{proof}
\end{corollary}

\par \textcolor{black}{We end this section by proving Bernstein type inequality for $V_A(\phi)$. Let $\mathcal{A}$ denote the class of continuously differentiable functions $\phi$ such that
\begin{itemize}
    \item[(i)] $|\phi(x)|\leq \frac{M_1}{|x|^{0.5+\epsilon}}$ and $|\phi^\prime(x)|\leq \frac{M_2}{|x|^{0.5+\epsilon}}$, for sufficiently large $x$, for some $M_1, M_2, \epsilon>0$.
    \item[(ii)] $ \text{ess sup}_{\omega\in I}\sum_{k\in \Z}(\frac{\omega+2kb\pi-p}{b})^2|\F_A\phi (\omega+2kb\pi)|^2<\infty$.
\end{itemize}
\begin{theorem}
Let $\phi\in \mathcal{A}$ be such that $\{T_k^A\phi:k\in\Z\}$ is a Riesz basis for $V_A(\phi)$. Then we have the Bernstein type inequality
\[
\|Bf\|^2\leq M\|f\|^2, ~~~\text{ for all }f\in V_A(\phi),
\]
where $Bf(x)=f^\prime(x)+\frac{iax}{b}f(x)$ and
 \[M=\text{ess sup}_{\omega\in I}\frac{\sum_{k\in\Z}(\frac{\omega+2kb\pi-p}{b})^2|\F_A\phi(\omega+2kb\pi)|^2}{\sum_{k\in\Z}|\F_A\phi(\omega+2kb\pi)|^2}.\]
\end{theorem}
\begin{proof}
Let $f(x)=\sum_{k\in \Z}c_kT_k^A\phi(x)$.\\
Then $f^\prime(x)=\sum_{k\in\Z}c_kT_k^A\phi^\prime(x)-\sum_{k\in \Z}\frac{iak}{b}c_kT_k^A\phi(x)$. Since $\phi\in\mathcal{A}$, the above equalities hold pointwise. Thus
\begin{align*}
Bf(x)&=f^\prime(x)+\frac{iax}{b}f(x)\\
&=\sum_{k\in \Z}c_kT_k^A\phi^\prime(x)-\sum_{k\in \Z}\frac{iak}{b}c_kT_k^A\phi(x)+\frac{iax}{b}\sum_{k\in \Z}c_kT_k^A\phi(x)\\
&=\sum_{k\in\Z}c_k\big(\frac{d}{dx}(T_k^A\phi(x))+\frac{iax}{b}T_k^A\phi(x)\big)\\
&=\sum_{k\in \Z}c_kBT_k^A\phi(x).
\end{align*}
Further, $\F_Af(\omega)=r(\omega)\F_A\phi(\omega)$, where $r(\omega)=\sum_{k\in\Z}c_k\rho_A(k)e^{-\frac{i}{b}k\omega}$. \textcolor{black}{Now, using $\F_A(Bf)(\omega)=i\frac{\omega-p}{b}\F_Af(\omega)$ (see Proposition 3.5 in \cite{H3}), we obtain}
\begin{align*}
\|Bf\|^2=\|\F_A(Bf)\|^2&=\|\F_A(\sum_{k\in \Z}c_kB(T_k^A\phi))\|^2\\
&=\int_\R|\sum_{k\in \Z}c_k\F_A(BT_k^A\phi)(\omega)|^2d\omega\\
&=\int_\R|\sum_{k\in\Z}c_k\frac{i(\omega-p)}{b}\F_A(T_k^A\phi)(\omega)|^2d\omega\\
&=\int_\R|\sum_{k\in\Z}c_k\rho_A(k)e^{-\frac{i}{b}k\omega}\frac{\omega-p}{b}\F_A\phi(\omega)|^2d\omega\\
&=\int_\R|r(\omega)|^2(\frac{\omega-p}{b})^2|\F_A\phi(\omega)|^2d\omega\\
&=\int_I |r(\omega)|^2\sum_{k\in\Z}(\frac{\omega+2kb\pi-p}{b})^2|\F_A \phi(\omega+2kb\pi)|^2d\omega\\
&\leq M\int_I|r(\omega)|^2\sum_{k\in\Z}|\F_A\phi(\omega+2kb\pi)|^2d\omega\\
&=M\int_\R|r(\omega)|^2|\F_A\phi(\omega)|^2d\omega=M\|f\|^2,
\end{align*}
\textcolor{black}{proving our assertion.}
\end{proof}}

\section{Sampling theorems}

\textcolor{black}{In this section, our aim is to obtain reconstruction formulae for the functions belonging to certain $V_A(\phi)$ from integer samples. We prove sampling formulae with $L^2$ convergence as well as uniform convergence. As a corollary, we obtain the result proved in \cite{Bhandari}, namely Shannon sampling theorem for the functions which are bandlimited in the SAFT domain. }


\begin{theorem} \label{exact_R}
Let $\phi\in W\big(C,\ell^1(\Z)\big)$ be such that $\{T_k^A\phi:k\in \Z\}$ forms a Riesz basis for $V_A(\phi)$. Then there is a function $S\in V_A(\phi)$ such that
\begin{equation}\label{reconstructionF}
f(t)=\sum_{n\in \Z} f(n)T_n^AS(t),
\end{equation}
for all $f\in V_A(\phi)$ if and only if $1/\phi_A^\dagger \in L^2(I).$
\begin{proof}
Assume that there is a $S\in V_A(\phi)$ such that (\ref{reconstructionF}) holds. Then $\phi(t) = \sum_{n\in \Z} \phi(n)T_n^AS(t)$. Taking SAFT on both sides we obtain 
\begin{equation} \label{eq1}
\F_A\phi(\omega)=\sum_{n\in \Z} \phi(n) \F_A (T_n^AS)(\omega)=\phi_A^\dagger (\omega)\F_AS(\omega).
\end{equation}
This implies that $supp(\F_A\phi)\subseteq supp(\phi_A^\dagger)$. Since $\phi_A^\dagger$ is $2\pi b$ periodic, we have $supp(T_{-2kb\pi}\F_A\phi)\subseteq supp(\phi_A^\dagger),~\forall~ k\in \Z.$\\
\par Since $\phi \neq 0$, we have $\F_A\phi \neq 0$ and hence $supp(\F_A\phi)$ is not a set of measure zero. We shall show that $\bigcup_{k\in \Z} supp(T_{-2kb\pi}\F_A\phi)=\R.$ If not, then there exists a $\Delta (\subseteq \R)$, a set of positive measure such that $\Delta \subset \R \setminus \bigcup_{k\in \Z} supp(T_{-2kb\pi}\F_A\phi)$ and $\F_A\phi (\omega +2kb\pi) =0, ~ \omega \in \Delta,~ \forall~ k.$ This in turn implies that $\sum_{k\in \Z}|\F_A\phi (\omega +2kb\pi)|^2=0$ on $\Delta $. This is a contradiction to our assumption that $\{T_k^A\phi :k\in\Z\}$ is a Riesz sequence. So $\phi_A^\dagger(\omega) \neq 0$ for almost every $\omega \in \R.$ Using (\ref{eq1}), we get $\F_AS(\omega)=\F_A\phi (\omega)/\phi_A^\dagger (\omega) $ and
\begin{align*}
\int_\R |S(t)|^2dt =\int_\R |\F_AS(\omega)|^2d\omega &= \int_\R |\frac{\F_A\phi(\omega)}{\phi_A^\dagger (\omega)}|^2d\omega\\
&= \int_I\frac{\sum_{k\in \Z} |\F_A\phi (\omega+2kb\pi)|^2}{|\phi_A^\dagger (\omega)|^2}d\omega \\
&\geq \|G_\phi^A\|_0\int_I\frac{1}{|\phi_A^\dagger (\omega)|^2}d\omega,
\end{align*}
where $\|G_\phi ^A\|_0 =\inf_{\omega \in I}\sum_{k\in \Z} |\F_A\phi (\omega+2kb\pi)|^2$. Since $\{T_k^A\phi : k\in\Z\}$ forms a Riesz sequence, $\|G_\phi^A\|_0>0$. Consequently $1/\phi_A^\dagger \in L^2(I).$
\par Conversely, assume that $1/\phi_A^\dagger \in L^2(I)$. Since $\{\frac{1}{\sqrt{2\pi |b|}}\rho_A(n)e^{-\frac{i}{b}n\omega} : n\in \Z\}$ forms an orthonormal basis for $L^2(I)$, there is a sequence $\{c_n\}\in \ell^2(\Z)$ such that
$$
\frac{1}{\phi_A^\dagger (\omega)}=\sum_{n\in \Z} c_n\rho_A(n)e^{-\frac{i}{b}n\omega}.
$$
Let $F(\omega)=\F_A\phi(\omega)/\phi_A^\dagger (\omega).$ Then
\begin{align*}
\int_\R |F(\omega)|^2d\omega &= \int_\R |\frac{\F_A\phi (\omega)}{\phi_A^\dagger (\omega)}|^2d\omega\\
&= \int_I\frac{\sum_{k\in \Z}|\F_A\phi(\omega+2kb\pi)|^2}{|\phi_A^\dagger|^2}d\omega\\
&\leq \|G_\phi ^A\|_\infty \|1/\phi_A^\dagger\|_{L^2(I)}^2,
\end{align*}
where $\|G_\phi^A\|_\infty = \sup_{\omega \in I}\sum_{k\in \Z} |\F_A\phi (\omega+2kb\pi)|^2$. Since $\{T_k^A\phi: k\in \Z\}$ is a Riesz sequence, $\|G_\phi^A\|_\infty <\infty$. Hence $F\in L^2(\R)$. Then there is exactly one $S\in L^2(\R)$ such that $\F_AS(\omega)=F(\omega)$. From the definition of $\phi_A^\dagger$ we get 
$$
\F_AS(\omega)=\F_A\phi(\omega)\sum_{k\in \Z}c_k\rho_A(k)e^{-\frac{i}{b}k\omega}=\sum_{k\in \Z}c_k\F_A(T_k^A\phi)(\omega),
$$
which shows that $S=\sum_{k\in \Z}c_kT_k^A\phi$ and hence $S\in V_A(\phi)$.
Now let $f\in V_A(\phi)$ with representation $f=\sum_{k\in \Z}a_kT_k^A\phi$. Taking SAFT leads to
\begin{align*}
\F_Af(\omega)&= \sum_{k\in \Z}a_k\F_A(T_k^A\phi)(\omega)\\
&= \sum_{k\in \Z}a_k\rho_A(k)e^{-\frac{i}{b}k\omega}\F_A\phi (\omega)\\
&= \F_AS(\omega)\phi_A^\dagger (\omega)\sum_{k\in \Z} a_k\rho_A(k)e^{-\frac{i}{b}k\omega}\\
&=\F_AS(\omega)\sum_{\ell \in \Z} \phi(\ell)\rho_A(\ell)e^{-\frac{i}{b}\ell 
\omega} \sum_{k\in \Z} a_k \rho_A(k)e^{-\frac{i}{b}k\omega}\\
&= \F_AS(\omega)\sum_{\ell \in \Z} \sum_{k\in \Z} a_k\phi(\ell) \rho_A(k)\rho_A (\ell)e^{-\frac{i}{b}(\ell+k)\omega}\\
&=\F_AS(\omega)\sum_{n\in \Z} \sum_{k\in \Z} a_k \phi (n-k)\rho_A(k)\rho_A(n-k)e^{-\frac{i}{b}n\omega}\\
&=\F_AS(\omega)\sum_{n\in \Z} \sum_{k\in \Z} a_k \phi(n-k)e^{-\frac{i}{b}k(n-k)}\rho_A(n)e^{-\frac{i}{b}n\omega}\\
&=\F_AS(\omega)\sum_{n\in \Z} \big(\sum_{k\in \Z} a_k T_k^A\phi(n)\big)\rho_A(n)e^{-\frac{i}{b}n\omega}\\
&= \sum_{n\in \Z}f(n)\F_A(T_n^AS)(\omega).
\end{align*}
The last equality finally gives $f=\sum_{n\in \Z} f(n)T_n^AS.$
\end{proof}
\end{theorem}

\textcolor{black}{In Theorem \ref{exact_R}, we obtained a sampling formula for functions belonging to $A$-shift invariant spaces with $L^2$ convergence. Now, our aim is to obtain another version of a sampling theorem where we obtain both $L^2$ convergence and pointwise convergence of the corresponding reconstruction formula. Towards this end, we prove the following}

\textcolor{black}{\begin{lemma}
Let $\phi\in L^2(\R)$. Then the following statements are equivalent.
\begin{itemize}
\item[(i)] For any $\{c_k\}\in\ell^2(\Z)$, the series of functions $\sum_{k\in\Z}c_kT_k^A\phi(t)$ converges to a continuous function.
\item[(ii)] $\phi\in C(\R)$ and $\sup_{t\in\R}\sum_{k\in\Z}|\phi(t-k)|^2<\infty$.
\end{itemize}
\end{lemma} 
\begin{proof}
Since $|T_k^A\phi(t)|=|T_k\phi(t)|=|\phi(t-k)|$, the proof follows as in Lemma 1 in \cite{ZhouSunJFAA}.
\end{proof}
For a sequence $\{c_k\}$, we define 
\[
\DF_A(\{c_k\})(\omega)=\frac{1}{\sqrt{2\pi|b|}}\sum_{k\in\Z}c[k]e^{\frac{i}{2b}(ak^2+2pk-2\omega k)}.
\]
For two sequences $\{c_k\}$ and $\{d_k\}$, we define
\[
(\{c_k\}*_A \{d_k\})[n]=\frac{1}{\sqrt{2\pi|b|}}\sum_{k\in \Z}c[k]T_k^Ad[n].
\]
Using the fact that $\{\frac{1}{\sqrt{2\pi |b|}}e^{\frac{i}{2b}(ak^2+2pk-2\omega k)}:k\in\Z\}$, is an orthonormal basis for $L^2(I)$, we get
\begin{equation*}
\|\DF_A\{c_k\}\|^2=\sum_{k\in\Z}|c[k]|^2.
\end{equation*}
Further, one can show that
\begin{equation*}
\DF_A\{\{c_k\}*_A\{d_k\}\}(\omega)=\DF_A(\{c_k\})(\omega)\DF_A(\{d_k\})(\omega).
\end{equation*}
Thus
\begin{equation}\label{eq:nsampling1}
\int_I|\DF_A(\{c_k\})(\omega)|^2|\DF_A(\{d_k\})(\omega)|^2d\omega = \frac{1}{\sqrt{2\pi|b|}}\sum_{n\in\Z}|\sum_{k\in\Z}c[k]T_k^Ad[n]|^2
\end{equation}}

\textcolor{black}{\begin{theorem}
Let $\{T_k^A\phi:k\in\Z\}$ be a frame sequence for $V_A(\phi)$. Then the following are equivalent.
\begin{itemize}
\item[(i)] The series $\sum_{k\in\Z}c_kT_k^A\phi(t)$ converges to a continuous function for any $\{c_k\}\in\ell^2(\Z)$ and there exists a frame $\{T_k^A\psi:k\in\Z\}$ for $V_A(\phi)$ such that
\begin{equation}\label{eq:nsampling 2}
f(t)=\sum_{k\in\Z}f(k)T_k^A\psi(t),~~\text{ for all }f\in V_A(\phi),
\end{equation}
where the convergence is both in $L^2(\R)$ and uniform on $\R$.
\item[(ii)] $\phi\in C(\R)$, $\sum_{k\in\Z}|\phi(t-k)|^2$ is bounded on $\R$ and
\begin{equation}\label{eq:nsampling 3}
m\chi_{E_\phi}(\omega)\leq |\phi_A^\dagger (\omega)|\leq M\chi_{E_\phi}(\omega),
\end{equation}
for some $m,M>0$, where $E_\phi:=\{\omega\in \R:w_\phi(\omega)\neq 0\}$, $w_\phi(\omega):=\sum_{k\in\Z}|\F_A\phi(\omega+2kb\pi)|^2$.
\end{itemize}
\end{theorem}
\begin{proof}
In order to show that (i) implies (ii), it is enough to show that \eqref{eq:nsampling 3} holds. Taking $f=\phi$ in \eqref{eq:nsampling 2}, we get
\[
\phi(t)=\sum_{k\in\Z}\phi(k)T_k^A\psi(t).
\]
Taking SAFT on both sides we obtain $\F_A\phi(\omega)=\phi_A^\dagger(\omega)\F_A\psi(\omega)$. This implies that $w_\phi(\omega)=|\phi^\dagger_A(\omega)|^2w_\psi(\omega)$, from which it follows that $E_\phi\subset E_\psi$. Since $\{T_k^A\phi:k\in\Z\}$ and $\{T_k^A\psi:k\in\Z\}$ are frame sequences, there exist $m,M>0$ such that
\[
m\leq |\phi_A^\dagger(\omega)|\leq M,~~\text{ a.e. }\omega\in E_\phi,
\]
using Theorem \ref{th: frame char}. We now show that $\phi_A^\dagger(\omega)=0$ a.e. on $I\setminus E_\phi$. To see this, take $c(\omega)=1-\chi_{E_\phi}(\omega).$ Then $c(\omega)=\sum_{k\in\Z}c_k\rho_A(k)e^{-\frac{i}{b}k\omega}$, for some $\{c_k\}\in\ell^2(\Z)$. Since
$c(\omega)\F_A\phi(\omega)=0$, taking inverse SAFT, we obtain $\sum_{k\in\Z}c_kT_k^A\phi(t)=0$, for all $t\in\R$. In particular, 
\[
\{\{c_k\}*_A\{\phi(k)\}\}[n]=\frac{1}{\sqrt{2\pi|b|}}\sum_{k\in\Z}c_kT_k^A\phi(n)=0.
\]
Thus using \eqref{eq:nsampling1}, we get
\begin{align*}
0=\sum_{n\in\Z}|\sum_{k\in\Z}c_kT_k^A\phi(n)|^2&=\int_I|\DF_A\{c_k\}(\omega)|^2|\DF_A\{\phi(n)\}(\omega)|^2d\omega\\
&=\int_{I\setminus E_\phi}|\phi_A^\dagger(\omega)|^2d\omega,
\end{align*}
which proves our claim.\\
Conversely assume that (ii) holds. Let
\[
\F_A\psi(\omega)=
\begin{cases}
\frac{1}{\phi_A^\dagger(\omega)}\F_A\phi(\omega)& \omega\in E_\phi\\
0, &\omega\notin E_\phi.
\end{cases}
\]
Then $\{T_k^A\psi:k\in\Z\}$ is a frame sequence by appealing to Theorem \ref{th: frame char}. Since $\frac{1}{\phi_A^\dagger}\in L^2(I)$, it can be easily seen that $\psi\in V_A(\phi)$. With the similar reasoning, we can say that $\phi\in V_A(\psi)$. Thus $V_A(\phi)=V_A(\psi)$. Now define,
\[
\F_A\tilde{\psi}(\omega)=
\begin{cases}
\frac{\overline{\phi_A^\dagger(\omega)}}{w_\phi(\omega)}\F_A\phi(\omega),&\omega\in E_\phi\\
0,&\omega\notin E_\phi.
\end{cases}
\]
We show that $\{T_k^A\tilde{\psi}:k\in\Z\}$ is the canonical dual of $\{T_k^A\psi:k\in\Z\}$. Let $S$ be the frame operator associated with the frame $\{T_k^A\psi:k\in\Z\}$. Since $S$ commutes with $T_k^A$, for all $k$, it is enough to show that $S\tilde{\psi}=\psi$. Consider
\begin{align*}
\F_A(S\tilde{\psi})(\omega)&=\sum_{k\in\Z}\langle \tilde{\psi},T_k^A\psi\rangle \F_A(T_k^A\psi)(\omega)\\
&=\sum_{k\in\Z}\rho_A(k)e^{-\frac{i}{b}k\omega}\F_A\psi(\omega)\langle \F_A\tilde{\psi},\F_A(T_k^A\psi)\rangle\\
&=\sum_{k\in\Z}\rho_A(k)e^{-\frac{i}{b}k\omega}\F_A\psi(\omega)\int_{\R} \F_A\tilde{\psi}(\eta)\overline{\F_A(T_k^A\psi)(\eta)}d\eta\\
&=\sum_{k\in\Z}\rho_A(k)e^{-\frac{i}{b}k\omega}\F_A\psi(\omega)\int_\R \F_A\tilde{\psi}(\eta)\overline{\rho_A(k)}e^{\frac{i}{b}k\eta}\overline{\F_A\psi(\eta)}d\eta\\
&=\frac{1}{2\pi |b|}\sum_{k\in\Z}e^{-\frac{i}{b}k\omega}\F_A\psi(\omega)\int_{E_\phi}e^{\frac{i}{b}k\eta}\frac{|\F_A\phi(\eta)|^2}{w_\phi(\eta)}d\eta\\
&=\F_A\psi(\omega)\frac{1}{2\pi|b|}\sum_{k\in\Z}e^{-\frac{i}{b}k\omega}\int_I \chi_{E_\phi}(\eta)e^{\frac{i}{b}k\eta}d\eta\\
&=\F_A\psi(\omega)\chi_{E_\phi}(\omega)=\F_A\psi(\omega),
\end{align*}
which proves our claim. Let $f\in V_A(\phi)$. Then $\F_Af(\omega)=r(\omega)\F_A\phi(\omega)$, where $r(\omega)=\sum_{k\in\Z}c_k\rho_A(k)e^{-\frac{i}{b}k\omega}$, for some $\{c_k\}\in\ell^2(\Z)$. For $k\in\Z$, consider
\begin{align*}
\langle f,T_k^A\tilde{\psi}\rangle&=\int_\R \F_Af(\omega)\overline{\F_A(T_k^A\tilde{\psi})(\omega)}d\omega\\
&=\frac{1}{2\pi|b|}\int_{E_\phi}r(\omega)|\F_A\phi(\omega)|^2\frac{\phi_A^\dagger(\omega)}{w_\phi(\omega)}\overline{\rho_A(k)}e^{\frac{i}{b}k\omega}d\omega\\
&=\frac{1}{2\pi|b|}\int_Ir(\omega)\phi_A^\dagger(\omega)\overline{\rho_A(k)}e^{\frac{i}{b}k\omega}d\omega\\
&=\int_I \DF_A\{c_k\}(\omega)\DF_A\{\phi(n)\}(\omega)\overline{\rho_A(k)}e^{\frac{i}{b}k\omega}d\omega\\
&=\int_I\DF_A\{\{c_k\}*_A\{\phi(n)\}\}(\omega)\overline{\rho_A(k)}e^{\frac{i}{b}k\omega}d\omega\\
&=\sqrt{2\pi|b|}\{\{c_k\}*_A\{\phi(n)\}\}=f(k).
\end{align*}
Hence
\[
f(t)=\sum_{k\in\Z}\langle f,T_k^A\tilde{\psi}\rangle T_k^A\psi(t)=\sum_{k\in\Z}f(k)T_k^A(\psi)(t).
\]
We notice that $\{f(k)\}\in\ell^2(\Z)$. Thus, in order to show the uniform convergence of the above series, it is enough to show that
\[
\sum_{k\in\Z}|T_k^A\psi(t)|^2=\sum_{k\in\Z}|\psi(t-k)|^2\leq M<\infty,~~\text{ for all }t\in\R,
\]
for some $M>0$. Since $\psi\in V_A(\phi)$, $\F_A\psi(\omega)=r(\omega)\F_A\phi(\omega)$, where $r(\omega)=\sum_{k\in\Z}c_k\rho_A(k)e^{-\frac{i}{b}k\omega}$, for some $\{c_k\}\in\ell^2(\Z)$. Thus $w_\psi(\omega)=|r(\omega)|^2w_\phi(\omega)$. Since $\{T_k^A\phi:k\in\Z\}$ and $\{T_k^A\psi:k\in\Z\}$ are frames for $V_A(\phi)$, $r$ is bounded on $E_\phi$. This implies that $\tilde{r}(\omega)=r(\omega)\chi_{E_\phi}(\omega)$ is bounded on $I$. Let $\tilde{r}(\omega)=\sum_{k\in\Z}\tilde{c_k}\rho_A(k)e^{-\frac{i}{b}k\omega}$, for some $\{c_k\}\in\ell^2(\Z)$. Since $r(\omega)\F_A\phi(\omega)=\tilde{r}(\omega)\F_A\phi(\omega)$, $\psi(t)=\sum_{k\in\Z}\tilde{c_k}T_k^A\phi(t)$. Thus
\begin{align*}
\sum_{n\in\Z}|\psi(t-n)|^2&=\sum_{n\in\Z}|\sum_{k\in\Z}\tilde{c_k}T_k^A\phi(t-n)|^2\\
&=\frac{1}{2\pi|b|}\int_I|\tilde{r}(\omega)|^2|\sum_{n\in\Z}\phi(t-n)\rho_A(n)e^{-\frac{i}{b}n\omega}|^2d\omega\\
&\leq \|\tilde{r}(\omega)\|_\infty \sum_{n\in\Z}|\phi(t-n)|^2,
\end{align*}
proving our assertion.
\end{proof}}

\par As a consequence of Theorem \ref{exact_R}, we obtain Shannon sampling theorem for the SAFT domain by taking $\phi=sinc,$ where 
$sinc(x)=\begin{cases}
          \frac{sin\pi x}{\pi x}, \quad &\text{if} \, x \neq 0 \\
          1, \quad &\text{if} \, x =0 \\
     \end{cases}.$
 We also write down the sampling theorem when $\phi$ is taken to be the second order symmetric $B$-spline. 
\begin{corollary}\label{SST}
Let $\phi=sinc $ and $\psi=C_{-\frac{a}{b}}\phi$. Then for every $f\in V_A(\psi)$, we have the following representation
\begin{equation} \label{SSTEQ}
f(t)=\sum_{k\in \Z}f(k)e^{i\frac{a}{2b}(k^2-t^2)}sinc(t-k),\quad t\in \R.
\end{equation}
\begin{proof}
We have 
\begin{align*}
\F_A\psi (\omega)&=\frac{\eta_A(\omega)}{\sqrt{2\pi|b|}}\int_\R C_{-\frac{a}{b}}\phi(t)e^{\frac{i}{2b}(at^2+2pt-2\omega t)}dt\\
&=\frac{\eta_A(\omega)}{\sqrt{|b|}}\widehat{\phi} (\frac{\omega - p}{b})\\
&=\frac{\eta_A(\omega)}{\sqrt{2\pi|b|}}\chi_{[-\pi,\pi]}(\frac{\omega-p}{b})\\
&= \frac{\eta_A(\omega)}{\sqrt{2\pi|b|}}\chi_{I+p}(\omega),
\end{align*}
from which it follows that $\sum_{k\in \Z} |\F_A\psi (\omega +2kb\pi)|^2=\frac{1}{2\pi |b|}$ for almost all $\omega \in I$, and this implies that $\{T_k^A\psi : k\in \Z\}$ is an orthonormal basis for the space $V_A(\psi)$. Moreover, we have $\psi_A^\dagger (\omega)=\sum_{k\in \Z}C_{-\frac{a}{b}}sinc(k)\rho_A(k)e^{-\frac{i}{b}k\omega}=1$. Consequently $1/\psi_A^\dagger = 1.$ So by Theorem \ref{exact_R} we have $\F_AS(\omega)=\F_A\psi(\omega),$ which implies that $S=\psi.$ Hence for $f\in V_A(\psi)$ we have 
\begin{align*}
f(t)&=\sum_{k\in \Z} f(k)T_k^A\psi(t) =\sum_{k\in \Z} f(k) e^{-i\frac{a}{b}k(t-k)}e^{i\frac{a}{2b}(t-k)^2}sinc(t-k)\\
&= \sum_{k\in \Z} f(k)e^{-i\frac{a}{2b}(t^2-k^2)}sinc(t-k).
\end{align*}
\end{proof}
\end{corollary}

\begin{corollary} \label{after_sst}
Let $\phi=sinc$. Let $\psi=C_{-\frac{a}{b}}\phi$. Then $V_A(\psi)=B_{I+p}^A$, where $B_{I+p}^A=\{f\in L^2(\R):supp\mathscr F_A(f)\subseteq I+p\}.$
\begin{proof}
It is clear from Corollary \ref{SST} that $V_A(\psi)\subseteq B_{I+p}^A$. Now we shall show that $B_{I+p}^A$ is a closed subspace of $L^2(\R)$ and the orthogonal complement of $V_A(\psi)$ in $B_{I+p}^A$ is zero, which will prove our assertion. As we have $\mathscr F_A(C_{-\frac{a}{b}}f)(\omega)=\frac{\eta_A(\omega)}{\sqrt{|b|}}\hat{f}(\frac{\omega-p}{b})$, $f\mapsto C_{-\frac{a}{b}}f$ is an isometry from $B_{[-\pi, \pi]}(:=\{f\in L^2(\R):supp \hat{f}\subset[-\pi, \pi]\})$ onto $B_{I+p}^A$. Therefore $B_{I+p}^A$ is a closed subspace of $L^2(\R)$. In Corollary \ref{SST} we have seen that $\{T_k^A\psi : k\in \Z\}$ is an orthonormal basis for $V_A(\psi)$. Consider $f\in B_{I+p}^A$ such that $\langle f, T_k^A\psi \rangle = 0,~~\forall~ k\in \Z$. Now we shall show that $f=0$. For all $k \in \Z$ we have
\begin{align*}
0=\langle f, T_k^A\psi \rangle &= \langle \mathscr F_A(f), \mathscr F_A(T_k^A\psi)\rangle \\
&=\int_{I+p} \mathscr F_A(f)(\omega)\overline{\rho_A}(k)e^{\frac{i}{b}k\omega}\overline{\mathscr F_A(\psi)(\omega)}d\omega \\
&= \frac{\overline{\rho_A}(k)}{\sqrt{2\pi |b|}}\int_{I+p}\mathscr F_A(f)(\omega)\overline{\eta_A(\omega)}e^{\frac{i}{b}k\omega}d\omega.
\end{align*}
Using the fact that $\{\frac{1}{\sqrt{2\pi |b|}}e^{-\frac{i}{b}k\omega}:k\in \Z\}$ is an orthonormal basis for $L^2(I+p)$, we get $\mathscr F_A(f)(\omega)\overline{\eta_A(\omega)}=0$ for a.e. $\omega \in I+p$, which in turn implies that $\|f\|_2=\|\mathscr F_A(f)\|_2=0$, proving our assertion.
\end{proof}
\end{corollary}

\par As a consequence of Corollary \ref{SST}, Corollary \ref{after_sst} can be restated as Shannon sampling theorem for the SAFT domain.  
\begin{corollary}
Let $f\in L^2(\R)$ be such that the $supp(\F_Af)\subseteq I+p$. Then the following sampling formula holds
$$
f(t)=\sum_{k\in \Z}f(k)e^{i\frac{a}{2b}(k^2-t^2)}sinc(t-k),\quad t\in \R.
$$
\end{corollary}

\begin{corollary}
Let $\phi = \chi_{[0,1]}\star \chi_{[0,1]}.$ Then for every $f\in V_A(\phi)$, we have the following reconstruction formula
$$
f(t)=\sum_{n\in \Z}f(n)T_n^A\phi (t),\quad t\in \R.
$$
\begin{proof}
Let $G$ be the Gramian associated with the sequence $\{T_k^A\phi :k\in \Z\}.$ Then $G_{j,k}=\langle T_k^A\phi, T_j^A\phi \rangle $ is a tridiagonal operator with all the diagonal elements, $d_i=1$. The elements above the diagonal are given by\\
 $u_j=
\begin{cases}
\frac{b^2}{ a^2}e^{i\frac{a}{b}j}(e^{-i\frac{a}{b}}(\frac{2ib}{a}-1) -(1+\frac{2ib}{a})), & a\neq 0\\
1/6 &a=0,
\end{cases}$\\
and the elements below the diagonal are given by \\$l_j=
\begin{cases}
-\frac{b^2}{ a^2}e^{-i\frac{a}{b}j }(e^{i\frac{a}{b}}(\frac{2ib}{a}+1) +(1-\frac{2ib}{a})), & a\neq 0\\
1/6 &a=0.
\end{cases}$\\
Notice that for $a=0,~G$ is strictly diagonally dominant and hence invertible, for $a\neq 0$ $|u_j|,~|l_j|$ are dominated by $\frac{4b^3}{ a^3}\sqrt{\frac{a^2}{4b^2}+1.}$ Now let $\frac{a}{2b}=r$ with $|r|\geq 1.2$, then $G$ is strictly diagonally dominant and hence invertible. In other words $\{T_k^A\phi:k\in \Z\}$ forms a Riesz basis for $V_A(\phi).$ Further $\phi_A^\dagger =1$ on $I$, from which the required assertion follows.
\end{proof}
\end{corollary}

\section{A local reconstruction method} As in the case of classical shift invariant space we can obtain a local reconstruction method for functions belonging to $V_A(\phi)$ with continuous generators satisfying polynomial decay from their samples. We state the results without the proof as the proofs follow similar lines. We refer to the works \cite{sarvesh}, \cite{Siva}.

\begin{propn}\label{lrp}
Let $\phi$ be a complex valued continuous function on $\R$ satisfying $\phi(x)=o(\frac{1}{|x|^\rho})~(\rho>1).$ Assume that $\{T_k^A\phi:k\in \Z\}$ is a Riesz basis for $V_A(\phi)$. Let $f\in V_A(\phi)$ and $[a^\prime,b^\prime]$ be an interval in $\R.$ Then for a given $\epsilon >0 $ there exist a positive integer $M$ and a sequence $c_f=\{c_k\}\in \ell^2(\Z)$ such that
$$
|f(x)-g_r(x)|<\|c_f\|_{\ell^2(\Z)}\frac{\epsilon}{N^\frac{\rho}{2}},
$$
for all $N\geq M$, for all $x\in [a^\prime,b^\prime]$ and $g_r(x)=\sum_{k\in[a^\prime-N+1,b^\prime+N-1]}c_kT_k^A\phi$. In other words $f|_{[a^\prime,b^\prime]}$ can be approximately determined by a finite number of coefficients $c_k$ locally. 
\end{propn}

\begin{theorem}
Fix $\rho \geq 2.$ Let $\phi$ be a complex valued continuous function on $\R$ satisfying $\phi(x)=o(\frac{1}{|x|^\rho}).$ Assume that $\{T_k^A\phi:k\in \Z\}$ is a Riesz basis for $V_A(\phi)$. Let $f\in V_A(\phi),~[a^\prime,b^\prime]$ be an interval in $\R$ and $\epsilon >0$. Let $M$ be a positive integer obtained in Proposition \ref{lrp}. Consider those points $x_j$ in the sample set $X$ such that $x_j\in [a^\prime,b^\prime]$. Let $(2M+b^\prime-a^\prime-1)\leq \#X\leq M^\rho$, where $\#X$ denotes the number of points in $X$. Define $U_{j,k}=T_k^A\phi(x_j),~1\leq j\leq \#X,~k\in [a^\prime-M+1,b^\prime+M-1]\cap \Z$. Then there exist $g_r\in V_A(\phi)$ such that
$$
\|f|_X-g_r|_X\|\leq \epsilon (1+\|U\|~\|U^\dagger\|) +\mathcal{O}(\epsilon ^2),
$$  
where $U^\dagger$ is the pseudoinverse of $U$. 
\end{theorem}
\subsection*{Experimental results} Local reconstruction method in the SAFT domain is implemented using Mathematica. We take $a=2,~b=3,~d=4,~p=q=0$ in the implementation. 
\par We take $\phi=\chi_{[-\frac{1}{2},\frac{1}{2}]}\star \chi_{[-\frac{1}{2},\frac{1}{2}]}$ and $f$, a linear combination of integer $A$-translates of $\phi.$ Here we reconstruct $f$ in the interval $[-10,10]$ from a sample set $X\subset [-10,10],~\#X=61,~M=10.$ We also plot the SAFT of $f$ and tabulate error for various values of $M$ in Table  \ref{tab:table}.

\begin{figure}
\centering
\begin{subfigure}{.50\textwidth}
  \centering
  \includegraphics[width=.50\linewidth]{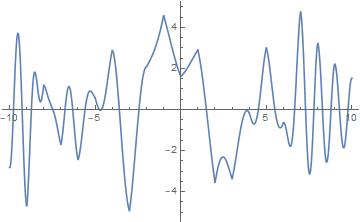}
  \caption{Real part}
\end{subfigure}%
\begin{subfigure}{.50\textwidth}
  \centering
  \includegraphics[width=.50\linewidth]{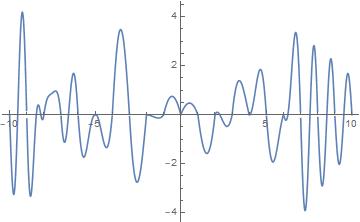}
  \caption{Imaginary part}
\end{subfigure}
\caption{Original signal}
\label{fig:original}
\end{figure}

\begin{figure}
\centering
\begin{subfigure}{.50\textwidth}
  \centering
  \includegraphics[width=.50\linewidth]{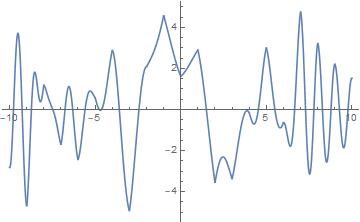}
  \caption{Real part}
\end{subfigure}%
\begin{subfigure}{.50\textwidth}
  \centering
  \includegraphics[width=.50\linewidth]{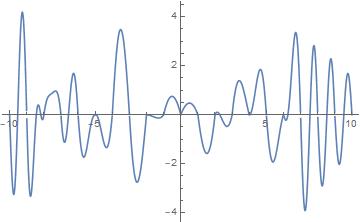}
  \caption{Imaginary part}
\end{subfigure}
\caption{Reconstructed signal}
\label{fig:reconstructed}
\end{figure}

\begin{figure}
\centering
\begin{subfigure}{.50\textwidth}
  \centering
  \includegraphics[width=.50\linewidth]{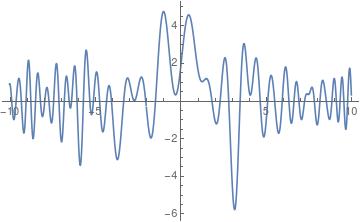}
  \caption{Real part}
\end{subfigure}%
\begin{subfigure}{.50\textwidth}
  \centering
  \includegraphics[width=.50\linewidth]{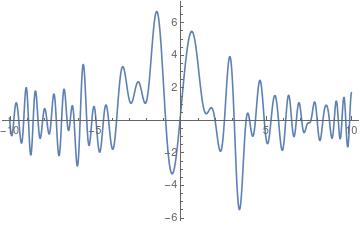}
  \caption{Imaginary part}
\end{subfigure}
\caption{SAFT of the signal in 
Figure \ref{fig:original}}
\label{fig:saft}
\end{figure}

\begin{table}[H]
  \begin{center}
    \caption{\\Computation of error}
    \label{tab:table}
    \begin{tabular}{c|c} 
      \textbf{$M$}  &  \textbf{Error}\\
      
      \hline
      
      $10$  & $5.74792\times 10^{-14}$\\
      $50$  & $5.69042\times 10^{-14}$\\
      $250$ & $5.50736\times 10^{-14}$\\
      $400$ & $4.7852\times 10^{-14}$
    \end{tabular}
  \end{center}
\end{table}

 \section*{Acknowledgement}
%
%
\par 

F.F. was partially supported by the Helmholtz Pilot Project Ptychography4.0 ZT-I-0025 and by the project  ZT-I-PF-4-024 within the Helmholtz Imaging Platform.

\bibliography{reference}
\bibliographystyle{amsplain}
\vspace{.02cm}
\textbf{Statements and Declarations}\\
1. All authors read and approved the final manuscript. There is no conflict of interest.\\
2. There is no associated data in the manuscript.

\end{document}